\documentclass[numbers,webpdf,imanum]{ima-authoring-template}

\usepackage[T1]{fontenc}
\usepackage{booktabs}
\usepackage[mathscr]{euscript}
\usepackage{color,amsmath,amssymb,mathrsfs}
\usepackage{graphicx}
\usepackage{mathtools}
\usepackage{amsfonts}
\usepackage{xspace}
\usepackage{float}

\makeatletter
\renewcommand*\l@subsubsection[2]{{}{}{}} %
\makeatother
\usepackage{subcaption}

\DeclareMathAlphabet{\mathsf}{OT1}{cmss}{m}{n}
\DeclareSymbolFont{operators}{OT1}{cmr}{m}{n}
\DeclareSymbolFont{letters}{OML}{cmm}{m}{it}
\DeclareSymbolFont{symbols}{OMS}{cmsy}{m}{n}
\DeclareSymbolFont{largesymbols}{OMX}{cmex}{m}{n}
\SetSymbolFont{operators}{bold}{OT1}{cmr}{bx}{n}
\SetSymbolFont{letters}{bold}{OML}{cmm}{b}{it}
\SetSymbolFont{symbols}{bold}{OMS}{cmsy}{b}{n}

\usepackage{enumitem}
\usepackage{mleftright}
\renewcommand{\left}{\mleft}
\renewcommand{\right}{\mright}

\newif\ifusecref
\usecreffalse %

\ifusecref
  \usepackage[nameinlink]{cleveref}
  \Crefname{equation}{}{}
  \Crefname{lemma}{Lemma}{Lemmas}
  \Crefname{remark}{Remark}{Remarks}
  \Crefname{cor}{Corollary}{Corollaries}
  \Crefname{prop}{Proposition}{Propositions}
  
\else
  \usepackage{zref-clever}
  \usepackage{zref-xr}
  \zcsetup{cap}
  \AddToHook{env/lemma/begin}{\zcsetup{countertype={theorem=lemma}}}
  \AddToHook{env/remark/begin}{\zcsetup{countertype={theorem=remark}}}
  \AddToHook{env/proposition/begin}{\zcsetup{countertype={theorem=proposition}}}
  \zcRefTypeSetup{equation}{Name-sg=,}
  \zcRefTypeSetup{setting}{Name-sg=Setting,}
  \let\Cref\zcref
\fi

\renewcommand\figurename{Figure}

\setlist[itemize]{labelsep=.6em, leftmargin=1.2em} 
\setlist[enumerate]{labelsep=.6em, leftmargin=*} 

\graphicspath{{./figs/}}
\ifpdf
  \DeclareGraphicsExtensions{.eps,.pdf,.png,.jpg}
\else
  \DeclareGraphicsExtensions{.eps}
\fi

\newif\ifshowcomments
\showcommentstrue

\newcommand{\commentout}[1]{}

\definecolor{grassgreen}{RGB}{92,135,39}

\DeclareMathAlphabet{\mathpzc}{OT1}{pzc}{m}{it}
\DeclareMathAlphabet{\mathup}{OT1}{\familydefault}{m}{n}

\newcommand{\alg}[1]{\textrm{\textsc{#1}}}

\newcommand{\mc}[1]{\mathcal{#1}}
\newcommand{\bs}[1]{\boldsymbol{#1}}
\newcommand{\mat}[1]{\mathbf{{#1}}}
\renewcommand{\vec}[1]{\bs{#1}}

\newcommand\restr[2]{{
  \left.\kern-\nulldelimiterspace %
  {#1}\vphantom{\big|} \right|_{#2}}}

\newcommand{\R}{\mathbb{R}}
\newcommand{\Z}{\mathbb{Z}}

\newcommand{\Exp}{\mathbb{E}}

\newcommand{\GM}[1]{\mathcal{N}\!\left( #1 \right)}

\DeclareMathOperator{\tr}{tr}
\DeclareMathOperator{\Var}{Var}
\DeclareMathOperator{\MSE}{MSE}

\counterwithin{algorithm}{section}
\DeclareFontShape{T1}{ptm}{m}{scsl}{<-> ssub * ptm/m/sc}{}
\DeclareFontShape{T1}{ptm}{m}{scit}{<-> ssub * ptm/m/sc}{}

\newcommand{\algname}{\alg{FlexTrace}\xspace}
\newcommand{\ialgname}{\textrm{i-}\alg{FlexTrace}\xspace}

\newcommand{\A}{\mat{A}}
\newcommand{\B}{\mat{B}}
\newcommand{\C}{\mat{C}}
\newcommand{\D}{\mat{D}}
\newcommand{\F}{\mat{F}}
\newcommand{\G}{\mat{G}}
\newcommand{\Hm}{\mat{H}}
\newcommand{\I}{\mat{I}}
\newcommand{\K}{\mat{K}}

\newcommand{\Q}{\mat{Q}}
\newcommand{\Rm}{\mat{R}}
\newcommand{\Sm}{\mat{S}}
\newcommand{\U}{\mat{U}}
\newcommand{\V}{\mat{V}}
\newcommand{\X}{\mat{X}}
\newcommand{\Xh}{\widehat{\X}}
\newcommand{\Y}{\mat{Y}}
\renewcommand{\Z}{\mat{Z}}

\newcommand{\Lm}{\bs{\Lambda}}
\newcommand{\Om}{\bs{\Omega}}
\newcommand{\Gmb}{\bs{\Gamma}}
\newcommand{\Sgm}{\bs{\Sigma}}
\newcommand{\om}{\vec{\omega}}
\newcommand{\Ah}{\widehat{\A}}

\newcommand{\Uh}{\widehat{\U}}
\newcommand{\Ut}{\widetilde{\U}}
\newcommand{\Lmh}{\widehat{\Lm}}
\newcommand{\Hmt}{\widetilde{\Hm}}

\DeclareMathSymbol{\mr}{\mathord}{AMSa}{"39}
\newcommand{\Nystrom}{Nystr\"om\xspace}
\newcommand{\pp}{\nolinebreak[4]\raisebox{.3ex}{\scriptsize\sf ++}}
\newcommand{\Nystrompp}{\alg{Nystr\"om\pp}\xspace}
\newcommand{\funNys}{\alg{FunNys}\xspace}

\newcommand{\funNystrompp}{\alg{funNystr\"om\pp}\xspace}
\newcommand{\Anys}{\Ah_{\mathrm{nys}}}

\newcommand{\Omi}{\Om_{\mr i}}
\newcommand{\Yi}{\Y_{\mr i}}
\newcommand{\Ahi}{\Ah_{\mr i}}

\newcommand{\AhOne}{\Ah_{\mr 1}}
\newcommand{\AhOneTwo}{\Ah_{\mr \{1, 2\}}}

\newcommand{\trh}{\widehat{\tr}}

\newcommand{\uinvnorm}[1]{\vert\kern-0.15ex\vert\kern-0.15ex\vert #1\vert\kern-0.15ex\vert\kern-0.15ex\vert}

\DeclareMathOperator*{\Cov}{Cov}

\DeclareMathOperator*{\diag}{diag}
\DeclareMathOperator*{\rank}{rank}

\newcommand{\mleq}{\preceq}

\newcommand{\cN}{\mathcal{N}}

\newtheorem{setting}{Setting}[section]

\newcommand{\Range}{\mathcal{R}\!\mathit{ange}}

\newcommand{\prcov}{\Gmb_{\text{pr}}}
\newcommand{\noisecov}{\Gmb_{\text{noise}}}
\newcommand{\postcov}{\Gmb_{\text{post}}}

\DeclareFontShapeChangeRule{it}{sc}{scsl}{scsl}
\newcommand{\nysprop}[1]{\hyperref[#1]{Lemma~\ref*{lem:nysprops}\ref*{#1}}}

\newenvironment{thmproof}[1]{\textit{Proof of \Cref{#1}.}}{\hfill\qed}
\newcommand{\AnysR}{\ensuremath{[\![\Anys]\!]_r}}
\newcommand{\Qh}{\widehat{Q}}

\graphicspath{{figs/}}

\theoremstyle{thmstyletwo}%
\newtheorem{theorem}{Theorem}[section]
\newtheorem{lemma}[theorem]{Lemma}
\newtheorem{proposition}[theorem]{Proposition}

\newtheorem{definition}{Definition}[section]

\numberwithin{equation}{section}

\begin{document}

\copyrightyear{2026}
\vol{00}
\pubyear{2026}
\access{Advance Access Publication Date: Day Month Year}
\appnotes{Paper}
\firstpage{1}

\title[FlexTrace: Exchangeable Randomized Trace Estimation]{FlexTrace: Exchangeable Randomized Trace Estimation for Matrix Functions}

\author{Madhusudan Madhavan*\ORCID{0009-0007-3135-2913}, Alen Alexanderian \ORCID{0000-0002-6371-6618}, and Arvind K.\ Saibaba\ORCID{0000-0002-8698-6100}
\address{\orgdiv{Department of Mathematics}, \orgname{North Carolina State University}} %
}

\authormark{M. Madhavan, A. Alexanderian, and A. K. Saibaba}

\corresp[*]{Corresponding author: \href{email:mmadhav@ncsu.edu}{mmadhav@ncsu.edu}}

\abstract{We consider the task of estimating the trace of a matrix function, $\tr(f(\A))$, of a large symmetric positive semi-definite matrix $\A$. This problem arises in multiple applications, including kernel methods and inverse problems. A key challenge across existing trace estimation methods is the need for matrix-vector products (matvecs) with $f(\A)$, which can be very expensive. In this article, we introduce a novel trace estimator, \algname, an exchangeable, single-pass method that estimates $\tr(f(\A))$ solely using matvecs with $\A$. We consider the case where $f$ is an operator monotone matrix function with $f(0)=0$, which includes functions such as $\log(1+x)$ and $x^{1/2}$, and derive probabilistic bounds showcasing the theoretical advantages of \algname. Numerical experiments across synthetic examples and application domains demonstrate that \algname provides substantially more accurate estimates of the trace of $f(\A)$ compared to existing methods.}

\maketitle

\section{Introduction}
\label{sec:intro}
Matrix functions are indispensable components of numerous scientific and engineering applications, including network analysis, partial differential equations (PDEs), and statistical learning \cite{EstradaNetwork2010,BenziMatrix2020,HighamFunctions2008}. A common task in these applications is the estimation of the trace of matrix functions \cite{UbaruSLQ2017,UbaruApplications2018}. For a symmetric positive-definite matrix $\A \in \mathbb{R}^{n \times n}$ and a function $f: [0, \infty) \to [0, \infty)$, the trace of the matrix function can be expressed in terms of the eigenvalues of $\A$, $\{\lambda_i\}_{i=1}^n$, as $\tr(f(\A)) = \sum_{i=1}^{n} f(\lambda_i)$. A fundamental challenge, however, is that the eigenvalues of $\A$ can be prohibitively expensive to obtain in large-scale settings, making the explicit computation of $\tr(f(\A))$ infeasible. Moreover, in many applications—such as PDE-constrained optimization or kernel methods—the matrix $\A$ itself is not explicitly available but only accessible through matvecs $\mathbf{x} \mapsto \A\mathbf{x}$, which poses a further challenge for the computation of the matrix function.

In this article, we introduce \algname, a single-pass method that efficiently estimates $\tr(f(\A))$ directly from matvecs with $\A$. We focus on the class of operator monotone matrix functions (see \Cref{ssec:notation}) that satisfy $f(0) = 0$. Operator monotone functions preserve spectral ordering of matrices and facilitate the derivation of theoretical bounds on $f(\A)$ using the spectral properties of $\A$. The estimation of $\tr(f(\A))$ for such functions is applicable to a variety of important applications. For example, the estimation of log-determinants is essential for applications including Gaussian processes \cite{DongScalable2017}, Bayesian inference \cite{AlexanderianEfficient2018}, and spatial statistics \cite{PaceSampling2009,avronfasterkernel}. Similarly, the matrix Schatten $p$-norm, defined for a rectangular matrix $\X$ as $\|\X\|_{(p)} := \tr((\X^\top \X)^{p/2})^{1/p}$, involves the estimation of the trace of a matrix function and is important to matrix completion among other low-rank recovery tasks \cite{LiSketching2013,NguyenLowRank2019,PanishMatrix2025}. Specifically, our method is applicable to the estimation of the Schatten $p$-norms for $1 \leq p \leq 2$ and perhaps the most important case corresponds to the nuclear norm ($p = 1$). We provide theoretical guarantees for $\algname$ for the important class of operator monotone functions (see \Cref{tab:apps}), and demonstrate its substantial accuracy and efficiency improvements over existing methods across different applications.

\subsection*{Related Work} 
Matrix functions have been well-studied in literature, with several works addressing theoretical and computational aspects \cite{HornMatrix2017,HighamFunctions2008}. The estimation of the trace of a matrix function can be considered a special case of implicit trace estimation (see \Cref{ssec:trace_estimation}). In the last decade, there has been increased focus to the use of randomized low-rank approximation techniques for trace estimation \cite{Saibaba2017RandomizedMT,Meyer2020HutchOS,Persson2021ImprovedVO}. These approaches utilize matvecs with random vectors to construct low-rank approximations of matrices \cite{GittensMahoney2013Nystrom,HalkoMartinssonTropp2011Survey,nakatsukasa_fast_2020}. However, such techniques often require matvecs with $f(\A)$, which is a nontrivial task. This is often addressed with polynomial approximations of $f$ or Krylov-subspace methods. There have been several important recent works in the latter direction, including Stochastic Lanczos Quadrature (SLQ) \cite{BaiLargescale1996,UbaruSLQ2017} and variants \cite{FuentesEstimating2023,ChenHallman23,YeonBOLT2025}.

In such approaches, Lanczos or block Lanczos methods \cite{HighamFunctions2008} are commonly used to approximate applications of $f(\A)$ to random vectors, which are then used to estimate the trace. However, such methods typically require multiple passes over the matrix $\A$, e.g., computations of the form $\A^2\bs{x} = \A(\A \bs{x})$. This poses significant restrictions on the applicability of these methods to problems that depend on offline computations with $\A$. For example, in resource-intensive applications, matvecs with $\A$ may require days/weeks of computation and cannot be accessed \textit{at will} within the trace estimation process. Moreover, the estimation of matvecs with $f(\A)$ is commonly done independently of the stochastic trace estimation, increasing the total number of matvecs required for estimation.

Our approach is closely related to and builds on the fun\Nystrom method \cite{PerssonKressner23}, which utilizes a randomized \Nystrom approximation of $\A$ to estimate $f(\A)$. However, we primarily focus on the task of estimating the trace of $f(\A)$, as opposed to $f(\A)$ itself, in a single-pass manner. The \funNystrompp method from \cite{PerssonKressner23} as well as the recent works \cite{ChenHallman23,YeonBOLT2025} outline techniques for matvec-efficient estimation of $\tr(f(\A))$, however, they require multiple passes over the matrix for proper utilization. The importance and role of exchangeablity, that is, invariance to permutation of random vectors, in the context of trace estimation was highlighted in \cite{Epperly2023XTraceMT}, which suggested a leave-one-out exchangeable trace estimation technique called \alg{XNysTrace} which we discuss in \Cref{ssec:impl_nys}. Our approach, \algname, builds on this idea to develop an exchangeable approach for the estimation of traces of matrix functions. Finally, we mention that our work also shares some similarities with the single-sample probing approach in \cite{CortinovisPreconditioned2026}, although we focus on developing single-pass trace estimators for operator monotone matrix functions.

\subsection*{Our Approach and Contributions}
In this article, we develop a novel estimator, \algname  (Function-agnostic Low-rank EXchangeable Trace estimation algorithm), that utilizes a low-rank randomized \Nystrom approximation and a Monte Carlo estimator of the remainder to approximate $\tr(f(\A))$. This approach avoids the need to compute or approximate matvecs with $f(\A)$ altogether, making it suitable for applications where matvecs with $\A$ is only accessible \textit{offline} (e.g., before the trace estimation process). Moreover, it is function-agnostic in that it enables the computation of $\tr(f(\A))$ for multiple functions without additional matvecs with $\A$. In addition to our proposed estimator \algname, we present \ialgname---an idealized variant of \algname that facilitates deeper theoretical analysis. The key contributions of this article include:
\begin{itemize}
    \item an exchangeable single-pass method, \algname, to compute the trace of a matrix function;
    \item theoretical guarantees for the expectation and mean-squared-error of the proposed algorithms and error bounds for the fun\Nystrom approximation \cite{PerssonKressner23} under any unitarily invariant norm;
    \item computationally efficient implementations of \algname with a focus on numerical stability; and
    \item numerical demonstration of the proposed algorithm for synthetic examples, Bayesian inference, matrix completion, and kernel methods.
\end{itemize}
This article is organized as follows. In \Cref{sec:background}, we define notation and provide the necessary background to set the stage for the proposed methods. In \Cref{sec:main}, we present the proposed algorithm, \algname, and its idealized variant \ialgname, highlight their theoretical properties, and provide an efficient and stable implementation of \algname. We provide numerical results in \Cref{sec:num_res}, highlighting the effectiveness of the proposed algorithms for synthetic matrices as well as domain applications such as Bayesian inversion and kernel methods. We provide the proofs of the key theoretical results in \Cref{sec:proofs}. Finally, we present concluding remarks in \Cref{sec:concl}.

\section{Background}
\label{sec:background}
In this section, we discuss the necessary background for randomized trace estimation of matrix functions.

\subsection{Key Assumptions and Notation}
\label{ssec:notation}
For the subsequent results and theorems, we make the following ground assumptions.

\paragraph*{Matrix and Eigenvalues} Throughout the article, $\A$ is assumed to be an $n \times n$ real SPSD matrix, with eigenvalues arranged in non-increasing order as $\lambda_1 \geq \lambda_2 \geq \cdots \geq \lambda_n \geq 0$. The spectral decomposition of $\A$ is represented as $\A = \U \Lm \U^T$, where $\Lm = \diag(\lambda_1, \dots, \lambda_n)$ and $\U \in \R^{n\times n}$ is an orthogonal matrix. We define the submatrices $\Lm_{i:j} := \diag(\lambda_i, \dots, \lambda_j)$ for indices $1 \leq i < j \leq n$. Given two symmetric matrices $\B, \C$, we say $\B \preceq \C$ if $\C-\B$ is SPSD. This defines a partial ordering on the set of symmetric matrices called the Loewner partial ordering \cite[p. 493]{HornMatrix2017}.

\paragraph*{Gaussian random matrix}  For $k \leq n$, $\Om \in \R^{n\times k}$ denotes a matrix with independent entries drawn from the standard normal distribution. In this article, we use $\cN(\bs{\mu}, \bs{\Sigma})$ to denote a Gaussian measure with mean $\bs{\mu}$ and covariance $\bs{\Sigma}$.

\paragraph*{Matrix Function} Given a function $f: \mathbb{R} \rightarrow \mathbb{R}$, the matrix function $f(\A)$ is defined using the spectral decomposition $\A = \U \Lm \U^T$, as  $f(\A) = \U f(\Lm)\U^T $,  where $f(\Lm) = \diag(f(\lambda_1),\dots,f(\lambda_n))$. A function $f$ is called operator monotone \cite[Chapter 5]{bhatia} if $\C \mleq \D$ implies $f(\C) \mleq f(\D)$. In this article, we consider functions in the set $\mathscr{F}$, defined as
\begin{equation}
    \mathscr{F} := \big\{f: [0, \infty) \to [0, \infty) :\quad \text{$f(0) = 0$ and $f$ is operator monotone}\big\},
    \label{eq:op_mon_set}
\end{equation}
A few important matrix functions in $\mathscr{F}$ are listed in \Cref{tab:apps}. 
\begin{table}[ht]
    \centering
    \caption{Some important examples of operator monotone functions}
    \begin{tabular}{l l p{4.5cm}} 
        \toprule
        \textbf{Function ($f$)} & \textbf{Equivalent form of $\text{tr}(f(\mathbf{A}))$} & \textbf{Applications} \\
        \midrule
        $\log(1+x)$ & $\log\det(\mathbf{I}+\mathbf{A})$ & Bayesian inference \cite{AlexanderianEfficient2018}, machine learning \cite{KramerGradients2024} \\
        \addlinespace
        $x^p$ ($0 \leq p \leq 1$) & $\text{tr}(\mathbf{A}^{p})$ & Schatten norm estimation \cite{LiSketching2013}, Rényi entropy \cite{DongEfficient2025} \\
        \addlinespace
        $\frac{x}{x+\zeta}$ ($\zeta > 0$) & $\text{tr}(\mathbf{A}(\mathbf{A} + \zeta\mathbf{I})^{-1})$ & Effective dimension \cite{avronfasterkernel}, kernel learning \cite{DongScalable2017} \\
        \bottomrule
    \end{tabular}
    \label{tab:apps}
\end{table}

\paragraph*{Norms} Throughout this article, we use $\|\cdot\|_2$ to denote the spectral norm and $\|\cdot\|_{(p)}$ to denote the Schatten $p$-norm for $p \geq 1$. For enhanced clarity, when referring to the Frobenius ($p=2$) and nuclear ($p=1$) norms, we utilize the notation $\|\cdot\|_F$ and $\|\cdot\|_*$, respectively. We denote an arbitrary unitarily invariant norm as $\uinvnorm{\A}$, which satisfies $\uinvnorm{\A} = \uinvnorm{\U\A\V}$ for orthogonal matrices $\U, \V$.

\subsection{Implicit trace estimation}
\label{ssec:trace_estimation}
In this section, we review randomized methods for implicit trace estimation, that is, the estimation of $\tr(\A)$ using matvecs with $\A$. A simple yet ubiquitous randomized trace estimation technique is the Girard-Hutchinson trace estimator \cite{Girard1989,Hutchinson1989ASE}. The key idea behind this estimator stems from \Cref{thm:HutchVar}.
\begin{theorem}
    \cite[Lemma 9]{AvronToledo2011Hutch}
    Let $\A \in \R^{n \times n}$ be SPSD and $\om \sim \cN(\bs{0}, \I)$. Then, 
    \begin{align*}
        \Exp[\om^\top \A \om] = \tr(\A), \quad\text{and} \quad
        \Var[\om^\top \A \om] = 2\|\A\|_F^2.
    \end{align*}
    \label{thm:HutchVar}
\end{theorem}
The Girard-Hutchinson estimator uses this property to estimate the trace with a Monte-Carlo approach. This involves drawing independent random vectors $\om_1, \om_2, \dots, \om_k$ from $\cN(\bs{0}, \I)$ and computing
\begin{equation}
    \trh_{GH} := \frac{1}{k}\sum_{i=1}^{k} \om_i^\top \A \om_i \approx \tr(\A).
    \label{eq:GH_estimator}
\end{equation} 
While this estimator is unbiased, its variance decays slowly at a rate of $\mc{O}(1/k)$. More recently, there have been several approaches \cite{Meyer2020HutchOS,Persson2021ImprovedVO,Epperly2023XTraceMT} that utilize randomized low-rank approximations to capture the (approximate) dominant eigenspace before employing this estimator. These methods provide $\mc{O}(1/k^2)$ variance decay, significantly outperforming the Girard-Hutchinson estimator for matrices with fast spectral decay. The randomized \Nystrom approximation (\Cref{ssec:Nystrom}) for SPSD matrices is particularly effective for such approaches.

\subsection{Randomized \Nystrom approximation}
\label{ssec:Nystrom}
The \Nystrom approximation \cite{GittensMahoney2013Nystrom} of $\A$ with test matrix $\Om$ is defined as,
\begin{equation}
    \Anys = \A\Om(\Om^\top \A\Om)^\dagger (\A\Om)^\top.
    \label{eq:nystrom}
\end{equation}
The randomized \Nystrom approximation constructs a low-rank approximation to $\A$ using a \textit{sketch} (a randomized embedding) of $\A$. This approximation is guaranteed \cite{chen_randomly_2025} to provide the \textit{best} rank-$k$ approximation of $\A$ whose range is contained by $\Range(\A\Om)$ and is majorized by $\A$ (i.e., $\Anys \mleq \A$) \cite[Theorem 5.3]{ando_schur_2005}. The \Nystrom approximation can be used to estimate the dominant eigenpairs of $\A$ efficiently; see \Cref{alg:Nystrom}.

\begin{algorithm}[ht]
    \caption{Low-rank approximation $[\Lmh, \Uh, \mathbf{C}, \Ut, \Z, \hat{r}] = \alg{NysSVD}(\A, \Om)$}
    \begin{algorithmic}[1]
        \itemsep0.25em
        \Require SPSD matrix $\A \in \R^{n \times n}$, test matrix $\Om \in \R^{n \times k}$
        \Ensure Diagonal $\Lmh \in \R^{k \times k}$ and $\Uh \in \R^{n \times k}$ with orthonormal columns such that $\Anys = \Uh\Lmh\Uh^\top$
        \State Form $\Y = \A\Om$
        \State Perform a thin QR decomposition, $\Y = \Q\Rm$
        \State Compute truncated pivoted Cholesky factor, $\C = \texttt{cholp}(\Om^\top \Y)$ where $\C \in \R^{k \times \hat{r}}$ with $\hat{r} \leq k$
        \State Perform backward solve, $\Z = \Rm\ /\ \C$
        \State Compute a thin SVD, $[\Ut, \Sgm, \sim] = \texttt{svd\_econ}(\Z)$
        \State Return $\Uh = \Q\Ut$ and $\Lmh = \Sgm^2$
        \State (optional) Return $\C, \Z, \Ut,$ and $\hat{r}$
    \end{algorithmic}
    \label{alg:Nystrom}
 \end{algorithm}\noindent

In practice, we avoid explicitly computing $(\Om^\top \A\Om)^\dagger$ as that is known to be numerically unstable. Instead, we utilize the \textit{Stabilized \Nystrom approximation} \cite{bucci_numerical_2025}, which utilizes a truncated pivoted Cholesky decomposition \cite[Section 10.3]{HighamAccuracy2002} to decompose $\Om^\top \A\Om = \C^\top \C$; this is denoted as \texttt{cholp} in line 3 of \Cref{alg:Nystrom}. We utilize the LAPACK \cite{LAPACK} routine \texttt{dpstrf}, which terminates the Cholesky decomposition once the largest remaining diagonal element is numerically zero; see \cite{LucasLapack2004,HammarlingLapack2004} for details.
This also enables quantifying the numerical rank of $\Anys$ which will be useful in \Cref{ssec:efficient}. See \cite{TroppWebber2023Survey,nakatsukasa_fast_2020,bucci_numerical_2025} for further details on the stability and implementations of the \Nystrom approximation.

\subsection{Implicit trace estimation with the \Nystrom approximation}
\label{ssec:impl_nys}
The randomized \Nystrom approximation of $\A$ provides a versatile and effective tool for trace estimation of low-rank or nearly low-rank matrices. 
Namely, $\tr(\Anys)$ can be used as an estimator for $\tr(\A)$ \cite{Saibaba2017RandomizedMT,PerssonKressner23}. Furthermore, utilizing this method alongside traditional randomized trace estimation techniques offers the potential for further performance. 
For example, we may consider a strategy where, rather than employing the 
entire sketch for the \Nystrom approximation, we reserve some vectors for a Girard-Hutchinson trace estimate of the residual. To formalize this construction, we let $\Omi$ and $\Yi$ denote the matrices $\Om$ and $\Y$ with the $i$-th column removed and define $\Ahi$ as the randomized \Nystrom approximation of $\mathbf{A}$ obtained with the test matrix $\Omi$ for $1 \leq i \leq n$. Subsequently, we may define a leave-one-out version of the \Nystrompp estimator \cite{Persson2021ImprovedVO}:
\begin{equation}
    \trh_{\rm Nys\pp} := \tr(\AhOne) + \om_1^\top (\A - \AhOne) \om_1.
    \label{eq:Nyspp_estimator}
\end{equation}
This typically offers a more balanced approach than using the low-rank approximation or Girard-Hutchinson estimator in isolation, and one may strategically allocate the matvecs utilized for each, depending on spectral properties of the matrix. 

The \Nystrom approximation can be further used to estimate matrix functions and their trace. Specifically, the \textit{fun\Nystrom} matrix approximation $f(\A) \approx f(\Anys)$ \cite{PerssonKressner23} estimates $f(\A)$ by applying the matrix function to the randomized \Nystrom approximation of $\A$. This can be efficiently computed using the eigendecomposition $\Anys = \Uh\Lmh\Uh^\top$ obtained in \Cref{alg:Nystrom} as $f(\Anys) = \Uh f(\Lmh) \Uh^\top$. Note that the fun\Nystrom approximation entirely avoids matvecs with $f(\A)$ and only requires matvecs with $\A$. This approach also enables the estimation of the trace, $\tr(f(\A)) \approx \trh_{\rm FN} := \tr(f(\Anys))$ \cite{PerssonKressner23,Saibaba2017RandomizedMT}. This estimator, which we call \funNys, is particularly effective for matrices $\A$ that are approximately low-rank. Similar to \Cref{eq:Nyspp_estimator}, this approach can also be used in conjunction with a Monte Carlo estimator with one sample vector, yielding the \funNystrompp estimator \cite{PerssonKressner23}:
\begin{equation}
    \trh_{\rm funNys++} := \tr(f(\AhOne)) + \om_1^\top (f(\A) - f(\AhOne)) \om_1.
    \label{eq:funNyspp_estimator}
\end{equation}
Although the \funNystrompp estimator is unbiased, it requires, unlike \funNys, the need for matvecs with $f(\A)$, which may not be accessible. We point out that both the \Nystrompp and \funNystrompp estimators are independent of the order of $(\om_2, \dots, \om_k)$, since $\AhOne$ is invariant to such order (see \nysprop{item:nys_invar}).

\subsection{Exchangeable trace estimation}

Although the \Nystrompp trace estimation algorithm is effective, it lacks a property called \textit{exchangeability}, that is, permutation-invariance with respect to random vectors $\{\om_i\}_{i=1}^k$. We make the definition of exchangeability precise in \Cref{def:exchangeable}.

\begin{definition}
    Let $Q : (\mathbb{R}^n)^k \to \mathbb{R}$ be a measurable function. 
    We say that $Q$ is exchangeable if for every permutation $\sigma \in S_k$
    and every $(\om_1,\dots, \om_k) \in (\mathbb{R}^n)^k$,
    \[
        Q(\om_1, \dots, \om_k)=Q(\om_{\sigma(1)}, \dots, \om_{\sigma(k)}),
    \]
    Here, $S_k$ denotes the symmetric group of degree $k$.
    \label{def:exchangeable}
\end{definition}
The importance of exchangeability is traced back to the \textit{exchangeability principle} \cite{HalmosTheory1946}, which states that the minimum-variance unbiased estimator must be exchangeable. The paper \cite{Epperly2023XTraceMT} utilizes this idea in the context of randomized linear algebra. In \Cref{thm:symmetrize}, we build on this insight and demonstrate that averaging an estimator over permutations of random vectors is guaranteed to provide an estimator with a lower mean-squared-error (MSE)\footnote{The MSE of an estimator $\hat{\theta}$ of a parameter $\theta$ is defined as $\MSE[\hat{\theta}] := \Exp (\theta - \hat{\theta})^2$.}. This is a key result that motivates our approach throughout the paper.

\begin{proposition}
    \label{thm:symmetrize}
    Let $\Qh(\om_1, \dots, \om_k)$ be an estimator for a quantity of interest $Q$, obtained with independent and identically distributed random vectors $\om_1, \dots, \om_k$, with finite second moments. Consider the symmetrized estimator, %
    \begin{equation}
        \overline{Q}(\om_1, \dots, \om_k) := \frac{1}{|S_k|} \sum_{\sigma \in S_k}\Qh(\om_{\sigma(1)}, \dots, \om_{\sigma(k)}),
        \label{eq:symmetrize}
    \end{equation}
    where $S_k$ denotes the symmetric group of degree $k$ with order $|S_k| = k!$. Then, $\overline{Q}$ is an exchangeable estimator with the same bias as $\Qh$ but lower variance. Namely, %
    \[(i)\ \Exp[\overline{Q}] = \Exp[\Qh]; \quad (ii)\ \Var[\overline{Q}] \leq \Var[\Qh]; \quad\text{and}\quad (iii)\ \MSE[\overline{Q}] \leq \MSE[\Qh].\]
\end{proposition}
\begin{proof}
    See \Cref{sec:proofs}.
\end{proof}

We observe that \Cref{thm:symmetrize} can be interpreted as a specific case of the Rao-Blackwell Theorem \cite[Theorem 7.3.17]{CasellaStatistical2002}, with the multiset of random vectors as a sufficient statistic, and a similar result for unbiased estimators has been observed in \cite[p. 320]{RaoLinear1973}. We highlight that unlike the exchangeability principle, \Cref{thm:symmetrize} applies to both unbiased and biased estimators. 
The recent work \cite{Epperly2023XTraceMT} uses the idea of exchangeability to develop the trace estimation technique, \alg{XNysTrace}, by symmetrizing the \Nystrompp estimator \Cref{eq:Nyspp_estimator}. Since the \Nystrompp estimator is already invariant to permutations of $(\om_2, \dots, \om_k)$ as observed in \Cref{ssec:impl_nys}, the symmetrization can be concisely expressed as:
\begin{equation}
    \trh_{\rm XN} := \frac{1}{k}\sum_{i=1}^{k} \left[\tr(\Ahi) + \om_i^\top (\A - \Ahi) \om_i\right].
    \label{eq:xnys_estimator}
\end{equation}
While \alg{XNysTrace} and similar methods enable efficient estimation of the matrix trace, such methods do not easily extend to estimating the trace of a matrix function without matvecs with $f(\A)$.

\section{Efficient trace estimation with \algname}
\label{sec:main}
In this section, we demonstrate the potential of exchangeable approaches to the estimation of the trace of matrix functions. In \Cref{ssec:overview}, we present two novel randomized trace estimation approaches for operator monotone functions of matrices, i--\algname and \algname. This is followed by theoretical analysis of these methods in \Cref{ssec:key_theory}. Finally, we provide an efficient and stable computational implementation for \algname in \Cref{ssec:efficient}. 

\subsection{Overview of approach}
\label{ssec:overview}
We first present an idealized estimator, \ialgname, which serves as a precursor and theoretical tool to guide the analysis of our proposed algorithm, \algname. \ialgname is defined by symmetrizing the \funNystrompp estimator. This symmetrization, similar to that of \Nystrompp, is simplified by utilizing the invariance of \funNystrompp to permutations of $(\om_2, \dots, \om_k)$. Namely, for independent random vectors $\om_1, \dots, \om_k \sim \GM{\bs{0}, \bs{I}}$, the \ialgname estimator can be expressed as,
\begin{equation}
    \trh_{\mathrm{iFT}} := \frac1k \sum_{i=1}^k\left(\tr f(\Ahi) + \om_i^\top (f(\A) - f(\Ahi))\om_i \right).
    \label{eq:iFT}
\end{equation}
Since \ialgname is a symmetrization of \funNystrompp, \ialgname is an exchangeable, unbiased estimator with lower variance than \funNystrompp by \Cref{thm:symmetrize}. However, note that \Cref{eq:iFT} is an idealized estimator, since computing the quadratic form $\om_i^\top f(\A)\om_i$ requires matvecs with $f(\A)$ for $1 \leq i \leq k$, which we wish to avoid. 

To eliminate this bottleneck, we propose \algname (\Cref{alg:practical}),  which replaces $f(\A)$ in the idealized algorithm with $f(\Anys)$. The estimator, \algname, is thus defined for independent random vectors $\om_1, \dots, \om_k \sim \GM{\bs{0}, \bs{I}}$ as,
\begin{equation}
    \trh_{\mathrm{FT}} := \frac1k \sum_{i=1}^k\left(\tr f(\Ahi) + \om_i^\top (f(\Anys) - f(\Ahi))\om_i \right).
    \label{eq:FT}
\end{equation}
Similar to \ialgname, \algname is also an exchangeable estimator; see \Cref{thm:flex_ex}. We outline the corresponding computational procedure for \algname in \Cref{alg:practical}. In the special case that $f$ is linear, $\om_i^\top f(\A)\om_i = \om_i^\top f(\Anys)\om_i$ for $1\le i \le k$, and both estimators, \ialgname ($\trh_{\mathrm{FT}}$) and \algname ($\trh_{\mathrm{iFT}}$), are the same as the \alg{XNysTrace} estimator ($\trh_{\rm XN}$, see~\eqref{eq:xnys_estimator}). Moreover, if $f$ is well-approximated by a quadratic function on the spectrum of $\A$, $\om_i^\top f(\A)\om_i$ is well-approximated by $\om_i^\top f(\Anys)\om_i$ for $ 1 \le i \le k$, thanks to \nysprop{item:nys_quad}. 

\begin{algorithm}[ht]
    \caption{\algname\!($\A, k$)}
    \begin{algorithmic}[1]
        \itemsep0.25em
        \Require SPSD matrix $\A \in \R^{n \times n}$, matvec budget $k$
        \Ensure Trace approximation $\trh_{\mathrm{FT}}$
        \State Generate standard Gaussian random matrix $\Om \in \R^{n \times k}$
        \State Form $\Y = \A\Om$ 
        \State Compute $\Anys = \Y(\Om^\top \Y)^\dagger \Y^\top$
        \For {$i = 1, \ldots, k$}
            \State Define $\Ahi := \Yi(\Omi^\top \Yi)^\dagger \Yi^\top$ \Comment{Remove $i$-th column of $\Om$ and $\Y$}
            \State Compute $\trh_i := \tr f(\Ahi) + \om_i^\top (f(\Anys) - f(\Ahi))\om_i$
        \EndFor
        \State Compute $\trh_{\mathrm{FT}} := \frac{1}{k}\sum_{i=1}^k \trh_i$
    \end{algorithmic}
    \label{alg:practical}
\end{algorithm}\noindent

\begin{theorem}
    The \algname estimator \Cref{eq:FT} is an exchangeable estimator, in the sense of \Cref{def:exchangeable}. 
    \label{thm:flex_ex}
\end{theorem}
\begin{proof}
    See \Cref{ssec:algname_proofs}.
\end{proof}

Since \algname does not require matvecs with $f(\A)$, it offers an efficient and practical way of estimating $\tr(f(\A))$. However, a na\"ive implementation of the method can be computationally inefficient and susceptible to numerical instability. We provide a more refined implementation of \algname in \Cref{ssec:efficient} that is more suitable for practical use.

We highlight two other features of the method that make \algname appealing compared to other approaches. First, the approach is \textit{single pass}, meaning that all the matvecs with $\A$ can be accessed simultaneously. This is beneficial in situations where repeated access to $\A$ is not available or has a latency. Second, it is \textit{function-agnostic} in that $\tr(f(\A))$ can be estimated efficiently for multiple matrix functions at once at negligible additional cost. This can be beneficial if the function $f_\theta$ is parameter-dependent and $\tr(f_\theta(\A))$ needs to be computed for multiple values of the parameter $\theta$. Such scenarios are common in spectral density estimation \cite{MattiStochastic2025} and kernel methods \cite{DongScalable2017}.

\subsection{Key Theoretical Results}
\label{ssec:key_theory}
In this section, we provide theoretical guarantees on the bias and MSE of the \ialgname and \algname estimators. We analyze the bias and MSE in \Cref{sssec:bias,sssec:variance} respectively, followed by an asymptotic interpretation of these bounds in \Cref{ssec:asymptotic}. Our goal is to analytically quantify the expected error of the trace estimators in terms of the spectral properties of the matrix $\A$ and the matrix function $f$. We adopt a two-pronged approach. First, we structurally bound the bias and MSE of the two proposed estimators in terms of the expected error of the fun\Nystrom approximation. Second, we derive upper bounds on that expected error in terms of the spectral properties of the matrix $\A$. It is noteworthy that theoretical results obtained for the latter are more broadly applicable, beyond the task of trace estimation. The proofs of all theorems in this section are provided in \Cref{sec:proofs}.

\subsubsection{Analysis of bias}
\label{sssec:bias}
First, we show that \ialgname is an unbiased estimator of the trace of $f(\A)$.
\begin{theorem}
    Consider an SPSD matrix $\A$ and let $f \in \mathscr{F}$. Then, the \ialgname estimator \Cref{eq:iFT} is an unbiased estimator of $\tr(f(\A))$.
    \label{thm:bias_iFT}
\end{theorem}
\begin{proof}
    See \Cref{ssec:ialgname_proofs}.
\end{proof}

Although \ialgname is unbiased for all functions, \algname is typically biased. \Cref{thm:bias_FT} states that the bias in \algname is bounded by the bias of a \funNys approximation with $(k - 1)$ matvecs, which, in turn, is bounded in terms of the trailing eigenvalues of $\A$. In particular, \Cref{thm:bias_FT} demonstrates that the bias in \algname can be small for matrices with low-rank or near- low-rank structure.

\begin{theorem}
    Consider an SPSD matrix $\A$ and let $f \in \mathscr{F}$. For $k \geq 4$, \algname (\Cref{alg:practical}) satisfies, 
    \[0 \leq \Exp\left[\tr(f(\A))-\trh_{\mathrm{FT}} \right] \leq \Exp\left[ \|f(\A) - f(\AhOne)\|_* \right] \leq (k-3)\|f(\Lm_{k\mr 2:n})\|_*.\]
    If $f$ is linear, then $\trh_{\mathrm{FT}}$ is unbiased for any $k \geq 1$.
    \label{thm:bias_FT}
\end{theorem}
\begin{proof}
    See \Cref{ssec:algname_proofs}.
\end{proof}

\Cref{thm:bias_FT} provides a guarantee for the trace estimator for \textit{any} function $f \in \mathscr{F}$. Nevertheless, we point out that this result can be improved for specific known functions; for example, the improved analysis of the fun\Nystrom approximation with respect to the matrix square root function provided in \cite[Section 3.3.3]{PerssonKressner23} can be utilized for tighter bounds.

\subsubsection{Analysis of MSE of estimators}
\label{sssec:variance}
Bounds for the MSEs of \ialgname and \algname estimators are presented in \Cref{thm:var_iFT} and \Cref{thm:var_FT}, respectively, in terms of the residual norm of the low-rank approximation to $f(\A)$. Note that the upper bounds for the MSE are also bounds for the variance of estimators, since the variance of an estimator is dominated by its MSE. For \ialgname, which is unbiased, the variance of the estimator is simply its MSE.
\begin{theorem}
    Consider an SPSD matrix $\A$ and let $f \in \mathscr{F}$. Then, the \ialgname  estimator \Cref{eq:iFT} satisfies,
    \[
        \MSE \left[\trh_{\rm iFT}\right] \leq \frac{2}{k} \, \Exp \left[ \| f(\A) - f(\AhOne) \|_F^2 \right]
        + 2 \, \Exp \left[ \| f(\AhOne) - f(\AhOneTwo) \|_F^2 \right],
    \]
    where $\AhOneTwo$ denotes the \Nystrom approximation of $\A$ with respect to the test matrix $\Om$ with the first two columns removed.
    \label{thm:var_iFT}
\end{theorem}
\begin{proof}
 See \Cref{ssec:ialgname_proofs}.   
\end{proof}
\begin{theorem}
    Consider an SPSD matrix $\A$ and let $f \in \mathscr{F}$. Then, the \algname (\Cref{alg:practical}) estimator satisfies,
    \[\MSE \left[\trh_{\rm FT}\right] \leq 2 \ \Exp\left[ \| f(\A) - f(\AhOne) \|_F^2 \right] +  2 \ \Exp \left[ \|f(\A) - f(\AhOne) \|_*^2 \right].\]
    \label{thm:var_FT}
\end{theorem}
\begin{proof}
 See \Cref{ssec:algname_proofs}.   
\end{proof}

Similar to the bound in \Cref{sssec:bias}, \Cref{thm:var_iFT,thm:var_FT} bound the variance of \ialgname and \algname in terms of the MSE of a rank-$(k\!-\!1)$ \Nystrom approximation. Observe, in particular, that the bound for the MSE of \algname is comparable to that of a \funNys trace approximation with one fewer matvec. 
\Cref{thm:var_iFT,thm:var_FT} bound the MSE of the estimators by terms involving the MSE of the fun\Nystrom approximation under the Frobenius and nuclear norms. This motivates us to derive a general unitarily invariant bound for fun\Nystrom. \Cref{thm:second_mom_bd_nys} provides a bound for this error under any unitarily invariant norm in terms of the trailing eigenvalues of $\A$ and the function $f$. 
\begin{theorem}
    Consider an SPSD matrix $\A$ and let $f \in \mathscr{F}$. For $k \geq 4$, we have
    \begin{align*}
        \Exp\left[\uinvnorm{f(\A)-f(\Anys)}^2\right]
        \leq 2\uinvnorm{f(\Lm_{k \mr 3:n})}^2 + \gamma(\Lm_{k \mr 3:n}, k) \uinvnorm{f(\Lm_{k \mr 3:n}^{1 / 2})}^2, 
        \intertext{where $\uinvnorm{\cdot}$ denotes any unitarily-invariant norm, and}
        \gamma(\D, k) := \frac{16(k-4)(k-2)}{3} \|\D\|_2 +  \frac{16e^4k^2}{125} (\|\D^{1/2}\|_*^2 + 2\|\D\|_*) + 2.
    \end{align*}
    \label{thm:second_mom_bd_nys}
\end{theorem}
\begin{proof}
See \Cref{ssec:app_prob_bd}
\end{proof}

Note that \Cref{thm:second_mom_bd_nys} has implications beyond trace estimation. Although \cite{PerssonKressner23} provides theoretical bounds on the approximation error of fun\Nystrom under different norms, the result in \Cref{thm:second_mom_bd_nys} is more generally applicable for the following reasons. First, the bound is applicable to any unitarily invariant norm, including the nuclear and Frobenius norms. Second, to our knowledge, \Cref{thm:second_mom_bd_nys} provides the probabilistic bounds for the squared norm error of the fun\Nystrom algorithm without assuming subspace iteration was done as in \cite[Theorem 3.4]{PerssonKressner23}.

We bound the MSEs of \algname and \ialgname in terms of the matrix function $f$ and the trailing eigenvalues of $\A$ in \Cref{thm:mse}. The MSE quantifies the expected squared error of the trace estimator and therefore assesses the overall quality of the estimator in a way that incorporates both its bias and variance.
\begin{theorem}
Consider an SPSD matrix $\A$ and let $f \in \mathscr{F}$. The MSE of the estimators $\trh_{\rm iFT}$, $\trh_{\rm FT}$, and $\trh_{\rm FN}$ are bounded as follows:
     \begin{align*}
        \MSE[\trh_{\rm iFT}] &\leq 
    \frac{k + 1}{k} \left(4\|f(\Lm_{k\mr 5:n})\|_F^2 + 2\gamma(\Lm_{k \mr 5:n}, k-2) \|f(\Lm_{k\mr 5:n}^{1 / 2})\|_F^2 \right), 
    \intertext{\hfill for $k \geq 6$; }
         \MSE[\trh_{\rm FT}] &\leq 4\left( \|f(\Lm_{k\mr 4:n})\|_F^2 + \|f(\Lm_{k\mr 4:n})\|_*^2 \right) \\
         &\qquad+ 2\gamma(\Lm_{k \mr 4:n}, k\!-\!1) \left( \|f(\Lm_{k\mr 4:n}^{1 / 2})\|_F^2 + \|f(\Lm_{k\mr 4:n}^{1 / 2})\|_*^2 \right),
    \intertext{\hfill for $k \geq 5$; and }
         \MSE[\trh_{\rm FN}] &\leq 2\|f(\Lm_{k\mr 3:n})\|_*^2 + \gamma(\Lm_{k \mr 3:n}, k) \|f(\Lm_{k\mr 3:n}^{1 / 2})\|_*^2,
     \end{align*}
     \hfill for $k \geq 4$,\\
     where $\gamma$ is as defined in \Cref{thm:second_mom_bd_nys}.
\label{thm:mse}
\end{theorem}
\begin{proof} 
    See \Cref{ssec:algname_proofs}.
\end{proof}

Observe in \Cref{thm:mse} that bounds for the root mean squared errors (RMSE) of both estimators are dominated by a factor of $k\|\Lm_{k\mr 5:n}^{1 / 2}\|_*\|f(\Lm_{k\mr 5:n}^{1 / 2})\|_*$. This suggests that the proposed algorithms perform better for matrices with fast eigenvalue decay, as expected.

\subsubsection{Asymptotic behavior of theoretical bounds}
\label{ssec:asymptotic}
To provide further insight, we highlight the asymptotic behavior of the bias and MSE of \algname (\Cref{thm:bias_FT,thm:mse}) in \Cref{tab:asympt_var} for a matrix with exponentially decaying eigenvalues of $\lambda_i \leq \alpha^i$ for $1 \leq i \leq n$, where $0 < \alpha < 1$, with respect to the number of matvecs $k$. For all choices of matrix functions analyzed, we observe exponential decay of the bias and variance. For example, for the function $\log(1+x)$, every additional matvec reduces the expected bias and RMSE by a factor of $\alpha$ asymptotically.

\begin{table}[ht]
    \centering
    \begin{tabular}{ccc}
        \toprule
        Matrix function ($f$) & Bias bound & MSE bound \\ 
        \midrule
        $x$ & 0 & $\mathcal{O}(k^2\alpha^{2k})$ \\ 
        $\log(1+x),\ \frac{x}{x+\zeta}$ & $\mathcal{O}(k\alpha^{k})$ & $\mathcal{O}(k^2\alpha^{2k})$ \\ 
        $\sqrt{x}$ & $\mathcal{O}(k\alpha^{k / 2})$ & $\mathcal{O}(\alpha^{k})$ \\ 
        \bottomrule
    \end{tabular}
    \caption{Asymptotic behavior of \algname estimator for $\lambda_i \leq \alpha^i$, where $0 < \alpha < 1$.}
    \label{tab:asympt_var}
\end{table}
We point out that our variance bounds for the case $f(x) = x$ exhibit the same asymptotic decay as the trace estimation bounds for \alg{XNysTrace} as derived by \cite{Epperly2023XTraceMT}. For the other functions in \Cref{tab:asympt_var}, \funNys and \ialgname also asymptotically adhere to the bias and MSE bounds displayed therein.

\subsection{Efficient and stable implementation of \algname}
\label{ssec:efficient}
Although the predominant cost of implicit trace estimation methods is often dominated by matvecs with $\A$, the computational scalability of the subsequent matrix computations cannot be neglected in some applications. In this section, we explore techniques to improve the computational efficiency and provide a faster and numerically stable implementation of \algname.

A key bottleneck of \Cref{alg:practical} is the need to compute $f(\Ahi)$ $k$ times. If done na\"ively, these operations can consume up to an additional $\mc{O}(nk^3 + k^4)$ floating point operations (flops), beyond the \Nystrom approximation itself. To avoid this, we make use of the following facts:\\
\begin{enumerate}[label=(\alph*)]
    \item If $\rank(\Anys) < k$, then $\A = \Anys$ almost surely for random $\Om$ \cite{PerssonKressner23}.\\ 
    \item If $\rank(\Anys) = k$, then
    \begin{equation}
        \Ahi = \Anys - \Y \left( \frac{(\Om^\top \Y)^{-1}\vec{e}_i\vec{e}_i^\top (\Om^\top \Y)^{-1}}{\vec{e}_i^\top (\Om^\top \Y)^{-1} \vec{e}_i} \right)\Y^\top,
        \label{eq:rankOneUpdate}
    \end{equation}
    where $\Y = \A\Om$ for $1 \leq i \leq k$ \cite[Theorem 14.2]{EpperlyMake2025}. \\ 
\end{enumerate}

By property (a), $\tr[f(\Anys)]$ is exact (almost surely) if $\rank(\Anys) < k$, in which case, no additional computations are needed. If $\rank(\Anys) = k$, we take advantage of property (b). Namely, using notation from \Cref{alg:Nystrom}, we express $\Anys = \Uh\Lmh\Uh^\top$, $(\Om^\top \Y)^{-1} = \C^{-1}\C^{-\top}$, and $\Y = \Q\Rm = \Q(\Ut\Ut^\top)\Rm = \Uh\Ut^\top\Rm$. Substituting these quantities into \Cref{eq:rankOneUpdate} yields
\begin{equation}
    \Ahi = \Uh(\Lmh - \vec{b}_i\vec{b}_i^\top)\Uh^\top, \quad \textrm{where }\vec{b}_i := \frac{\Ut^\top\Rm \C^{-\top}\vec{e}_i}{\|\C^{-\top}\vec{e}_i\|_2}.
    \label{eq:rank_one_b}
\end{equation}
Therefore, we have $f(\Ahi) = \Uh f(\Lmh - \vec{b}_i\vec{b}_i^\top)\Uh^\top$ for $1 \leq i \leq k$.
A key observation is that to enable the use of $f(\Ahi)$, we only need to compute the application of $f$ to $\Lmh - \vec{b}_i\vec{b}_i^\top$, which is a $k\times k$ diagonal-plus-rank-one (DPR1) matrix. We make use of highly efficient techniques to compute the eigendecomposition of DPR1 matrices in $\mc{O}(k^2)$ flops \cite{Bunch78DPR1, Gu94DPR1}, compared to the standard $\mc{O}(k^3)$ flops for dense matrices. Once the eigendecomposition $\Lmh - \vec{b}_i\vec{b}_i^\top = \V_i\Lmh_i\V_i^\top$ is obtained for $1 \leq i \leq k$, the remaining steps for \algname naturally follow:
\begin{align*}
    \trh_{FT}
    &= \frac{1}{k}\sum_{i=1}^{k} \tr[f(\Ahi)] + \om_i^\top (f(\Anys) - f(\Ahi))\om_i \\
    &= \frac{1}{k}\sum_{i=1}^{k} \tr[f(\Lmh_i)] + \vec{x}_i^\top(f(\Lmh_i) - \V_if(\Lmh_i)\V_i^\top) \vec{x}_i,
\end{align*}
where $\vec{x}_i := \Uh^\top \om_i$ for $1 \leq i \leq k$. Note that this process entirely avoided the explicit computation of $f(\Ahi)$ for $1 \leq i \leq k$. 

We thus present the accelerated \algname algorithm in \Cref{alg:fast_version}.
A few comments regarding \Cref{alg:fast_version} are in order:
\begin{itemize}
    \item Line 3: Computing $\rank(\Anys)$ must be treated carefully in finite-precision arithmetic. In our implementation, we utilize the numerical rank $\hat{r}$ obtained from the truncated pivoted Cholesky in \alg{NysSVD} (\Cref{alg:Nystrom}). The numerical rank is also more conservative, since it avoids additional computations when $\Om^\top \Y$ is ill-conditioned.
    \item Line 7: Note that $\X = \Uh^\top \Om$ is simplified as $\X = \Ut^\top \Z^{-\!\top} \mathbf{C}$, by substituting $\Uh = \Q\Ut = \Y\Rm^{-1}\Ut$ and $\Z = \mathbf{R}\mathbf{C}^{-1}$.
    \item Line 9: The DPR1 eigendecomposition can be done in $\mc{O}(k^2)$ flops; e.g., see \cite[Algorithm 1]{Stor15DPR1} or the LAPACK routine \texttt{dlaed9} \cite{LAPACK}.
\end{itemize}
~
\begin{algorithm}[ht]
    \caption{Efficient version of \algname($\A, k$)}
    \begin{algorithmic}[1]
        \itemsep0.25em
        \Require SPSD matrix $\A \in \R^{n \times n}$, matvec budget $k$, $f \in \mathscr{F}$
        \Ensure Trace approximation $\trh_{\mathrm{FT}}$
        \State Generate matrix $\Om \in \R^{n \times k}$ with i.i.d standard normal entries
        \State Obtain $[\Lmh, \Uh, \mat{C}, \Ut, \Z, \hat{r}]$ = $\alg{NysSVD}(\A, \Om)$ 
        \If {$\hat{r} < k$} %
            \State Return $\trh_{\mathrm{FT}} = \tr\left[ f(\Lmh) \right]$
        \EndIf \vspace{0.2in}
        \State Define $\mat{S} \in \R^{k \times k}$ column-wise as $\vec{s}_i = \mat{C}^{-\top}\vec{e}_i/\|\mat{C}^{-\top}\vec{e}_i\|_2$ for $1 \leq i \leq k$
        \State Compute $\B = \Ut^\top \Z \mat{S}$ and $\X = \Ut^\top \Z^{-\!\top} \mat{C}$ 
        \For{$i = 1, \ldots, k$}
            \State Compute eigendecomposition, $\Lmh - \vec{b}_i\vec{b}_i^\top = \V_i\Lmh_i\V_i^\top$ %
            \State Compute $\trh_i := \tr\left[f(\Lmh_i)\right] + \vec{x}_i^\top f(\Lmh) \vec{x}_i - \vec{x}_i^\top \V_i f(\Lmh_i) \V_i^\top \vec{x}_i$
        \EndFor
        \State Compute $\trh_{\mathrm{FT}} := \frac{1}{k}\sum_{i=1}^k \trh_i$
    \end{algorithmic}
    \label{alg:fast_version}
\end{algorithm}

We compare the computational costs of the na\"ive and accelerated implementations of \algname in \Cref{tab:comp_cost}.
\begin{table}[ht]
    \centering
    \begin{tabular}{cccc} \toprule
        \textbf{Method} & \textbf{Matvecs with $\A$} & \textbf{Cost to compute $f(\Anys)$} & \textbf{Additional cost} \\\midrule 
        \Cref{alg:practical} & $k$ & $\mc{O}(k^2(n+k))$ & $\mc{O}(nk^3 + k^4)$ \\
        \Cref{alg:fast_version} & $k$ & $\mc{O}(k^2(n+k))$ & $\mc{O}(k^3)$ \\\bottomrule
    \end{tabular}
    \caption{Computational cost of \algname with \Cref{alg:practical,alg:fast_version} (flops)}
    \label{tab:comp_cost}
\end{table}
It is also worthwhile to note that the accelerated implementation (specifically, lines 6-12 in \Cref{alg:fast_version}) allows for parallelization in three different ways. First, the matrix sketch can be computed in parallel for different random vectors. Second, $\trh_i$ can be computed in parallel for $1 \leq i \leq k$. Third, the DPR1 structure (line 9, \Cref{alg:fast_version}) allows for the computation of individual eigenpairs independently, enabling further in-loop parallelization and storage savings. 
The accelerated implementation, \Cref{alg:fast_version}, is also numerically stable in practice as it avoids the explicit computation of $(\Om^\top \Y)^{-1}$.

\section{Numerical results}
\label{sec:num_res}
In this section, we demonstrate the performance of \algname via several numerical experiments. In \Cref{ssec:synth_mats}, we apply \algname to a suite of synthetic matrices to analyze its effectiveness. This is followed by an application of \algname to nuclear norm estimation in \Cref{ssec:nuclear_norm}, Bayesian inverse problems in \Cref{ssec:bayes}, and kernel methods in \Cref{ssec:kernel}. 

\subsection{Application to synthetic matrices}
\label{ssec:synth_mats}
In this section, we compare \algname (\Cref{alg:fast_version}) with existing approaches on synthetic matrices of size $1000\times 1000$ each. To this end, consider matrices of the form $\A = \U\Lm\U^\top$ with the following eigenvalue choices (for $1 \leq i \leq 1000$):
\begin{itemize}
    \item \texttt{Flat}: $\lambda_i = 3 - 2\frac{i-1}{n-1}$
    \item \texttt{Poly}: $\lambda_i = i^{-2}$
    \item \texttt{Exp}: $\lambda_i = 0.9^{i-1}$
    \item \texttt{Step}: $\lambda_i = \{1: i \leq 50, \ 10^{-3} : i \geq 51\}$
\end{itemize}
In \Cref{sssec:funNysCompare}, we compare the performance of \algname to \funNys on these representative matrices. Subsequently, we analyze \algname in comparison to multi-pass trace estimation techniques in \Cref{sssec:multi_pass}. Unless otherwise stated, all errors displayed in this section are averaged over $100$ independent trials and shaded regions represent the $90\%$ interpercentile range. 
\subsubsection{Comparison to \funNys}
\label{sssec:funNysCompare}
In this section, we compare the performance of the proposed \algname estimator with the \funNys approach on the synthetic matrices. The \funNys serves as a useful baseline for the proposed approach, as both methods are single-pass estimators for $\tr(f(\A))$. We perform two experiments on these synthetic matrices. First, we fix the matrix function and compare the accuracy of the trace estimators on the aforementioned matrices. Second, we compare the estimators for different matrix functions $f \in \mathscr{F}$, for a chosen matrix.

In \Cref{fig:log_synth_mats},  we display the relative error of \algname compared to \funNys for estimating $\tr(f(\A))$, where $f(\A) := \log(\I + \A)$ for the matrices with different spectral profiles. Across all four spectral profiles, we observe that \algname consistently outperforms the baseline \funNys by approximately one to two orders of magnitude in terms of average relative error. In the \texttt{Exp} case (\Cref{subfig:exp_log}), both methods exhibit linear convergence on a log-linear scale, adhering with the theory in \Cref{ssec:asymptotic}. However, it is still worth noting that \algname maintains a consistent edge over \funNys. The performance advantage of \algname compared to \funNys is most visible in the \texttt{Step} and \texttt{Poly} matrices, where \algname outperforms \funNys by several orders of magnitude. This suggests that \algname has a better tendency to capture small trailing eigenvalues (long tails) than \funNys. For the \texttt{Flat} matrix, where the eigenvalue decay is minimal, both methods exhibit little performance improvement with additional matvecs. Nevertheless, \algname consistently maintains a lower relative error compared to \funNys.

\Cref{fig:log_synth_mats} also displays shaded uncertainty regions for both \funNys and \algname. Since the y-axis is log-scaled, the shaded regions for \funNys appear to be much smaller than those for \algname. However, with the exception of \Cref{subfig:flat_log}, the interpercentile regions of \algname and \funNys are of the same magnitude for all the synthetic examples shown. In \Cref{subfig:poly_log}, for example, the interpercentile region for \algname is $15\%$ smaller than that of \funNys on average over the matvec budget. In the case of \texttt{Flat}, however, we do observe that \algname exhibits higher variance than \funNys when the spectrum of $\A$. This matches our theoretical expectations; for example, if $\A = \I$ and $f(x) = x$, \funNys always returns the inaccurate estimate $\trh_{FN} = k$ with zero variance, whereas \algname produces an unbiased estimate of the trace albeit with higher variance.
These results empirically validate the efficiency of \Cref{alg:fast_version} in capturing trace information more effectively than \funNys across a variety of spectral regimes.
\begin{figure}[ht]
    \centering
    \begin{subfigure}[b]{0.35\linewidth}
        \centering
        \includegraphics[width=\linewidth]{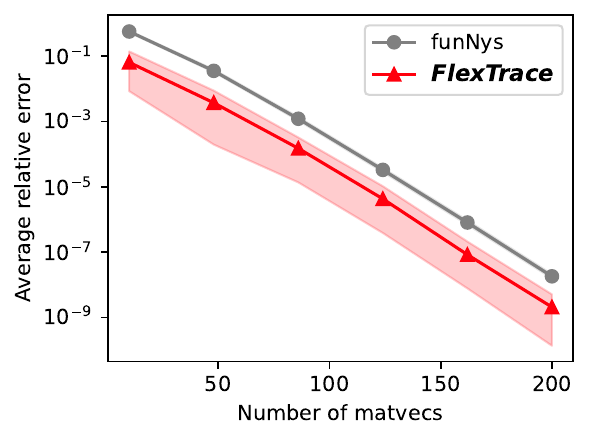}
        \caption{\texttt{Exp} matrix}
        \label{subfig:exp_log}
    \end{subfigure}
    \begin{subfigure}[b]{0.35\linewidth}
        \centering
        \includegraphics[width=\linewidth]{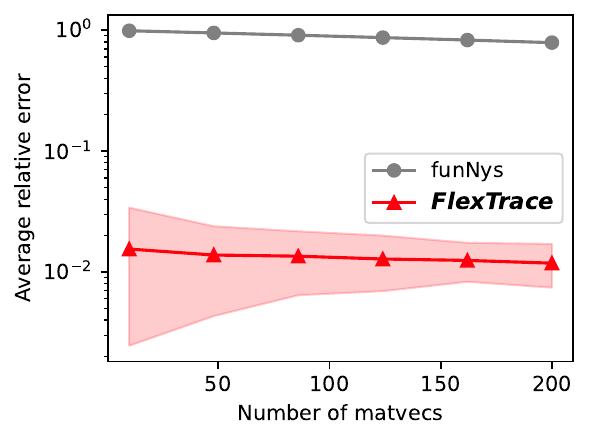}
        \caption{\texttt{Flat} matrix}
        \label{subfig:flat_log}
    \end{subfigure}
    \begin{subfigure}[b]{0.35\linewidth}
        \centering
        \includegraphics[width=\linewidth]{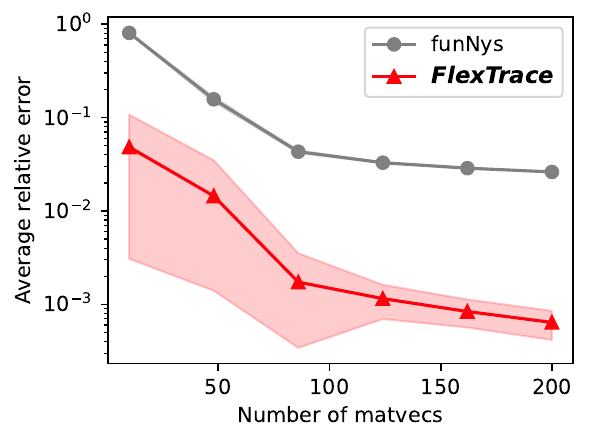}
        \caption{\texttt{Step} matrix}
        \label{subfig:step_log}
    \end{subfigure}
    \begin{subfigure}[b]{0.35\linewidth}
        \centering
        \includegraphics[width=\linewidth]{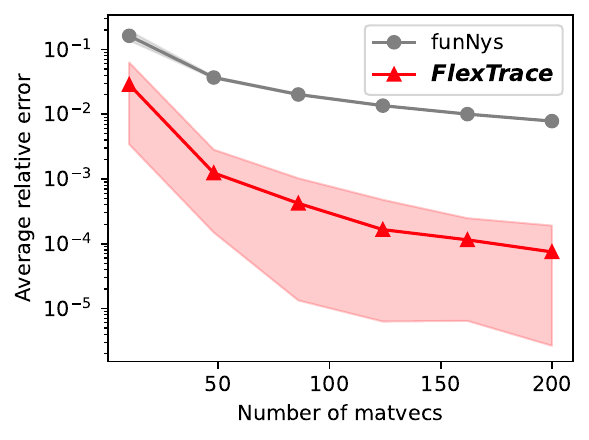}
        \caption{\texttt{Poly} matrix}
        \label{subfig:poly_log}
    \end{subfigure}
    \caption{Demonstration of \algname (\Cref{alg:fast_version}) on synthetic test matrices.}
    \label{fig:log_synth_mats}
\end{figure}

In \Cref{fig:poly_synth_mats}, we compare the performance of \algname to \funNys on a fixed matrix \texttt{Poly} but on several functions  $ x,\ x/(1+x),\log(1+x)$, and $\sqrt{x}$. We observe in \Cref{fig:poly_synth_mats} that \algname consistently outperforms \funNys across various operator monotone matrix functions. For the functions $ x,\ x/(1+x),$ and $\log(1+x)$, we observe that \algname has a lower relative error than \funNys by more than an order of magnitude, with the former achieving $10^{-4}$ relative accuracy with only $200$ matvecs. On the other hand, we note that both methods perform poorly for $f(x) = \sqrt{x}$, since even though $\A$ itself exhibits fast spectral decay, $f(\A) = \A^{1/2}$ does not. In comparison, for the other three functions, the spectral decay of $f(\A)$ is comparable to the spectral decay of $\A$.

\begin{figure}[ht]
    \centering
    \begin{subfigure}[b]{0.35\linewidth}
        \centering
        \includegraphics[width=\linewidth]{synthetic_new/flextrace_poly_log.pdf}
        \caption{$f(x) = \log(1+x)$}
        \label{subfig:log_poly}
    \end{subfigure}
    \begin{subfigure}[b]{0.35\linewidth}
        \centering
        \includegraphics[width=\linewidth]{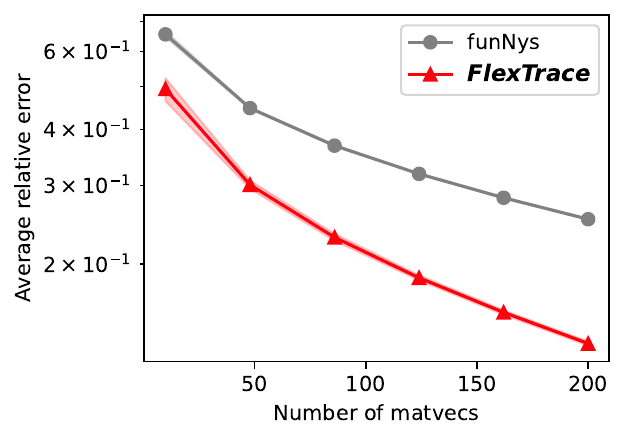}
        \caption{$f(x) = \sqrt{x}$}
        \label{subfig:sqrt_poly}
    \end{subfigure}
    \begin{subfigure}[b]{0.35\linewidth}
        \centering
        \includegraphics[width=\linewidth]{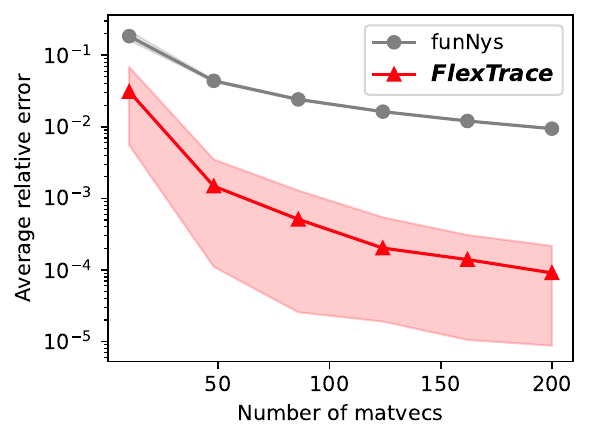}
        \caption{$f(x) = x/(x+1)$}
        \label{subfig:effdim_poly}
    \end{subfigure}
    \begin{subfigure}[b]{0.35\linewidth}
        \centering
        \includegraphics[width=\linewidth]{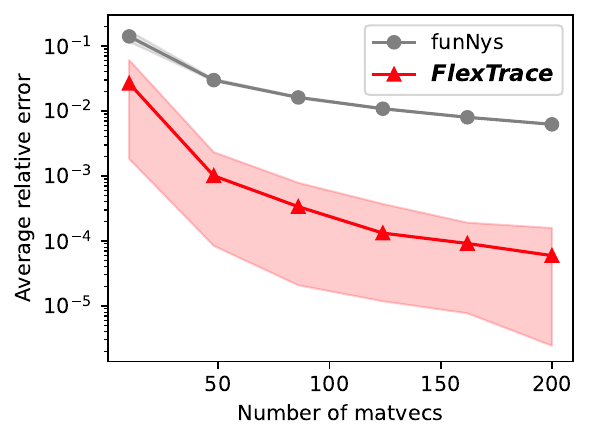}
        \caption{$f(x)=x$}
        \label{subfig:trace_poly}
    \end{subfigure}
    \caption{Demonstration of \algname (\Cref{alg:fast_version}) on the \texttt{Poly} matrix.}
    \label{fig:poly_synth_mats}
\end{figure}

\subsubsection{Comparison to multi-pass techniques} 
\label{sssec:multi_pass}
Although \algname is designed to be a single-pass method, it is useful to compare \algname (\Cref{alg:fast_version}) to multi-pass approaches to gain additional insight. In particular, we consider the following multi-pass methods for trace estimation of matrix functions:
\begin{enumerate}
    \item[(i)] \funNystrompp : A hybrid approach proposed by \cite{PerssonKressner23} to estimate $\tr(f(\A))$. This method uses $f(\Anys)$ to approximate $f(\A)$ and estimates the trace of the residual matrix using the Hutchinson method.
    \item[(ii)] \alg{SLQ}: Stochastic Lanczos Quadrature, as advanced by \cite{UbaruSLQ2017}, that integrates Gauss quadrature with Hutchinson's trace estimation. 
    \item[(iii)] \alg{KA-STE}: Krylov-aware stochastic trace estimation, as proposed by \cite{ChenHallman23}, that exploits properties of the Krylov subspace to enable higher efficiency for computing trace of matrix functions.
\end{enumerate}
Observe that methods (i)-(iii) involve approximating matvecs with $f(\A)$ using the Lanczos method. Unlike \algname, these techniques require applying $\A$ in a serial or blockwise serial manner, and often require computationally expensive procedures such as re-orthogonalization for numerical stability. A comparison between \algname and the methods outlined above is further complicated due to the fact that methods (i)-(iii) require the specification of several additional hyperparameters (block sizes, Lanczos iterations, allocation of matvecs, etc.). We attempt to provide a fair comparison by adhering to the hyperparameter selection strategies advised for \alg{KA-STE}, \funNystrompp, and \alg{SLQ} in the respective papers \cite{PerssonKressner23,UbaruSLQ2017,ChenHallman23} to the extent possible. We outline these choices in \Cref{tab:hyperparams}, where $k$ denotes the matvec budget. Note that we do not perform additional subspace iteration in \funNystrompp. For \alg{KA-STE}, if the matvec budget is insufficient to accommodate the attempted number of Lanczos iterations, the latter is reduced as necessary to satisfy the budget. 

\begin{table}[ht]
    \centering
    \begin{tabular}{l p{0.7\textwidth}c}
        \toprule
        \textbf{Method} & \textbf{Hyperparameter Settings} (for matvec budget $k$) & \textbf{Ref.} \\
        \midrule
        \textbf{\algname} & --- \\ 
        \funNys & --- \\ 
        \alg{SLQ} & Lanczos iterations: 5 &\cite{UbaruSLQ2017}\\ 
        \funNystrompp & \Nystrom rank: $k/2$; Lanczos iterations: 5 &\cite{PerssonKressner23}\\ 
        \alg{KA-STE (1)} & Block size: 2; Krylov depth: $k/4$; Lanczos iterations: 10 &\cite{ChenHallman23}\\ 
        \alg{KA-STE (2)} & Block size: 4; Krylov depth: $k/8$; Lanczos iterations: 5 &\cite{ChenHallman23}\\ 
        \bottomrule
    \end{tabular}
    \caption{Hyperparameter choices utilized in generating \Cref{fig:compare_log_poly}}
    \label{tab:hyperparams}
\end{table}

In \Cref{fig:compare_log_poly}, we plot the estimation errors for the estimation of $\tr(f(\A))$ using the aforementioned methods, as well as \funNys and \algname. Specifically, we focus on two specific cases: the \texttt{Exp} matrix with $f(x) = \log(1+x)$ and the \texttt{Flat} matrix with $f(x) = \sqrt{x}$. As detailed in \Cref{sssec:funNysCompare}, these two scenarios correspond to situations where $f(\A)$ exhibits fast and slow spectral decay, respectively, and are hence chosen as representative examples. We do not display the interpercentile regions in \Cref{fig:compare_log_poly} to keep the visual clear. 
Observe in \Cref{fig:compare_log_poly} (left) that for the \texttt{Exp} matrix, \algname achieves the lowest relative error among all the methods compared. On the other hand, for the \texttt{Flat} matrix, we observe that the multi-pass methods, notably \alg{SLQ}, generally produce more accurate estimates than \algname. These results suggest that \algname can still be valuable in applications where multi-pass methods can be applied and competitive with multi-pass trace estimation techniques for matrices that exhibit sufficiently fast spectral decay.

\begin{figure}[ht]
    \centering
    \includegraphics[width=0.38\textwidth]{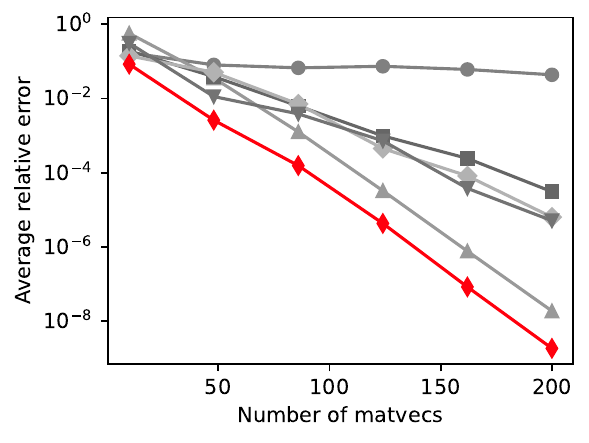}
    \includegraphics[width=0.52\textwidth]{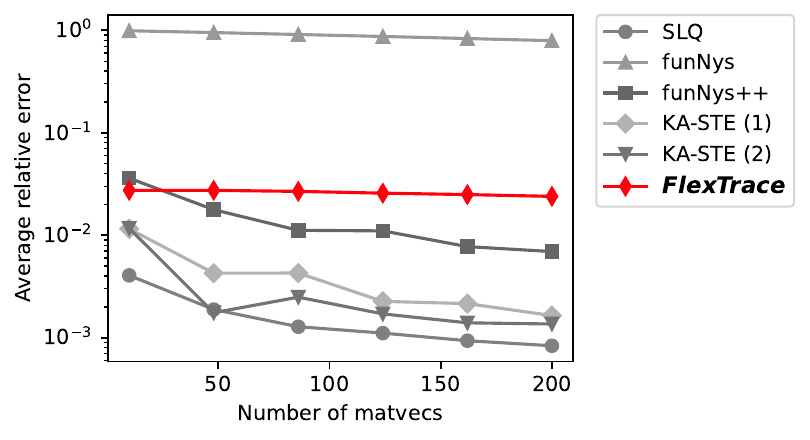}
    \caption{Comparison of \algname with existing methods for $f(x)=\log(1+x)$ with \texttt{Exp} (left) and $f(x) = \sqrt{x}$ with \texttt{Poly} (right).}
    \label{fig:compare_log_poly}
\end{figure}
\subsection{Application to Nuclear Norm Estimation}
\label{ssec:nuclear_norm}
In this section, we demonstrate the effectiveness of \algname for estimating the nuclear norm. The nuclear norm plays a key role in matrix completion---the problem of estimating missing entries of a partially observed matrix. Specifically, one of the most common approaches to matrix completion is nuclear-norm-minimization \cite{CandesExact2009,NguyenLowRank2019,PanishMatrix2025}, in which missing entries are completed such that the nuclear norm of the matrix is minimized. A famous application of matrix completion is the \textit{Netflix problem} \cite{bennett2007netflix}, where user–movie ratings are used to produce useful recommendations. We consider the \texttt{MovieLens 100k} dataset \cite{HarperMovieLens2015}, which is a benchmark for recommender systems and consists of 100,000 ratings (1 to 5) for 1682 movies from 943 users. We let $\mat{X} \in \R^{943 \times 1682}$ denote the rating matrix with rows corresponding to users and columns corresponding to movies, where missing entries are filled with zeros. Note that the $\X$ is a sparse matrix which contains approximately 94\% zero entries. Our goal is to estimate the nuclear norm of $\X$, $\|\X\|_* = \tr((\X\X^T)^{1 / 2})$. 

Computation of the nuclear norm typically requires the  computation of the singular value decomposition (SVD), which can be computationally expensive for large matrices. An alternative is to use the randomized SVD algorithm \cite{HalkoMartinssonTropp2011Survey}, which can estimate the leading singular values of $\X$ using matvecs with $\X$ and $\X^\top$. The sum of the estimated singular values can be used to approximate the nuclear norm of $\X$. We denote this approach as \alg{RandSVD}. On the other hand, we may utilize trace-estimation approaches such as \algname or \funNys to estimate $\|\X\|_* = \tr(f(\A))$ where $f(x) = \sqrt{x}$ and $\A = \X\X^\top$. We compare these approaches in \Cref{fig:matrix_completion}. Namely, we plot the relative errors obtained with each approach over the total number of matvecs (with both $\X$ and $\X^\top$).

\begin{figure}[ht]
    \centering
    \includegraphics[width=8cm]{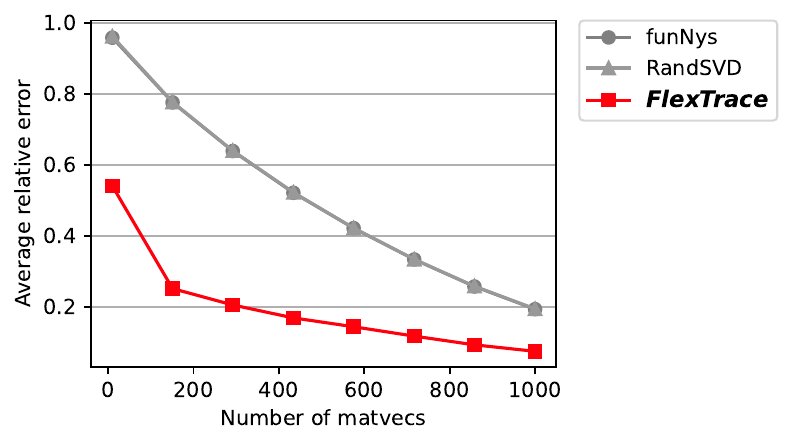}
    \caption{Comparison of \algname, \funNys, and \alg{RandSVD} for nuclear norm estimation of $\X$}
    \label{fig:matrix_completion}
\end{figure}

As observed in \Cref{fig:matrix_completion}, \algname results in superior efficiency compared to both \funNys and \alg{RandSVD}. Note that \funNys and \alg{RandSVD} are nearly indistinguishable in \Cref{fig:matrix_completion} since both methods estimate the leading $k$ eigenvalues of $\X\X^\top$ and thus perform similarly. On the other hand, \algname obtains with 300 matvecs nearly the same accuracy of \alg{RandSVD} with 1000 matvecs. This result suggests that \algname can be a valuable tool for efficient nuclear norm estimation in practical contexts.

\subsection{Application to Bayesian inverse problems}
\label{ssec:bayes}
In this section, we demonstrate the applicability of \algname in the context of Bayesian inversion. To this end, we consider the following advection-diffusion equation,
\begin{equation}\label{eq:adv_diff_prob}
\begin{alignedat}{2}
    \frac{\partial u}{\partial t} - \kappa \Delta u + \bs{v} \cdot \nabla u = f &&\qquad\text{in } \mc{D} \times (0, T),\\
    u(0, \bs{x}) = m(\bs{x}) &&\qquad\text{in } \mc{D},\\[0.5ex]
    \nabla u \cdot \bs{n} = 0 &&\qquad\text{on } \partial\mc{D}  \times (0, T), 
  \end{alignedat}
\end{equation}
where $\mc{D} \subseteq \R^3$ is a bounded domain, 
$u: \mc{D} \to \R$ is the state variable, $f: \mc{D} \to \R$ is a source term, $\bs{v}: \mc{D} \to \R^3$ is a divergence-free velocity field, and $\kappa > 0$ is a diffusion coefficient. The inverse problem aims to recover the unknown initial condition $m$ from measurement data of $u$ at the final time $T$.

We consider a Bayesian framework for this inverse problem. The infinite-dimensional formulation, finite-element discretization, and numerical solution of such problems have been well-studied in \cite{Bui-Thanh2013}, under the assumption of a Gaussian prior and additive Gaussian noise model, which we also adopt here. To keep the presentation simple, we focus on the discretized version of the problem. Namely, we wish to estimate $\bs{m} \in \R^n$ from pointwise observations $\bs{y} \in \R^d$ of the terminal state $u(\cdot, T)$. This relation is specified by the data model, 
\begin{equation}
    \bs{y} = \F\bs{m} + \bs{\eta}, \quad \bs{\eta} \sim \mc{N}(\bs{0}, \noisecov),
    \label{eq:fwd_model}
\end{equation}
where $\F \in \R^{d \times n}$ is a discretized parameter-to-observable map and $\noisecov = \sigma^2 \I \in \R^{d \times d}$ is the noise covariance matrix. We consider a prior distribution $\mu_{\text{pr}} = \mc{N}\left(\bs{0}, \prcov\right)$ where $\prcov$ is obtained by discretizing the squared inverse of an elliptic differential operator \cite{Bui-Thanh2013}, with scale parameter $\alpha = 1$ and correlation length parameter $\beta=5$. It is well known that, in this setting, the posterior distribution of the state variable $\bs{m}$ is given by $\mu_{\text{post}}^{\bs{y}} = \mc{N}(\bs{m}_{\text{\tiny MAP}}, \postcov)$, where 
\[\postcov = \left( \F^*\noisecov^{-1} \F + \prcov^{-1}\right)^{-1}, \quad \bs{m}_{\text{\tiny MAP}} = \postcov\F^*\noisecov^{-1} \bs{y}.\]

In this numerical example, we let $\mc{D} = (0, 1)^3$ and consider the discretized parameter dimension $n = 100,\!000$. The velocity field is set as $\bs{v}([x_1, x_2, x_3]) := -(0.5-x_2, x_1-0.5, 0)^\top$ and the terminal time as $T = 1.0$. We generate synthetic data by manufacturing a ground-truth initial condition, $\bs{m}_{\rm true}$, as shown in \Cref{fig:bayes_ic} (left). We solve the forward model \Cref{eq:fwd_model} using $\bs{m}_{\rm true}$ to obtain observation data $\bs{y}$ on a uniform grid of $d = 216$ sensor locations, as visualized in \Cref{fig:bayes_ic} (right). To simulate measurement error, we add 1\% noise ($\sigma = 0.01\|\bs{y}\|_2/\sqrt{d} $) to the observations $\bs{y}$.

\begin{figure}[ht]
    \centering
    \includegraphics[width=0.35\textwidth]{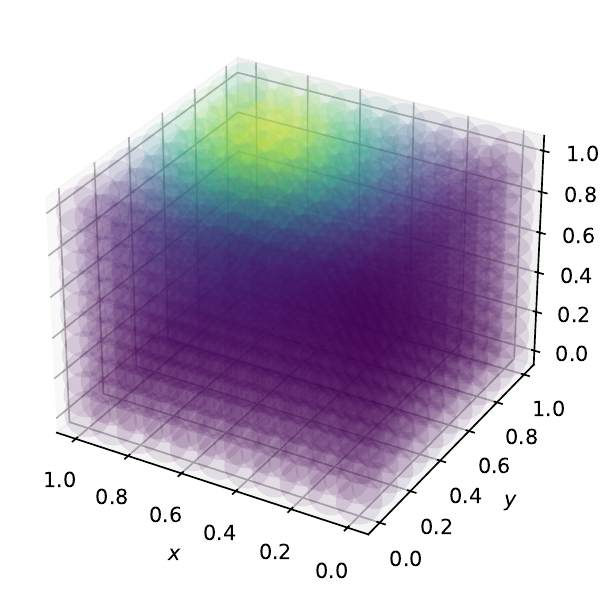}
    \includegraphics[width=0.4\textwidth]{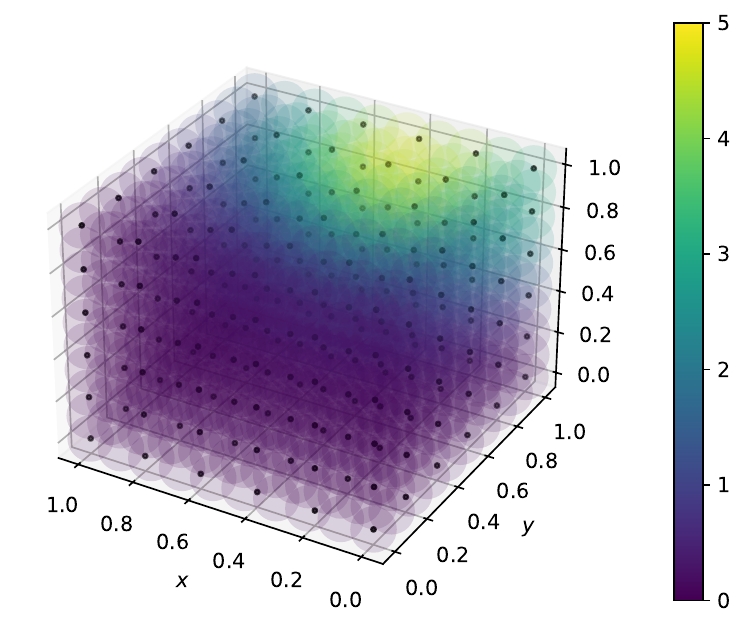}
    \caption{Ground-truth initial condition $\bs{m}_{\rm true}$ (left) and terminal state (right)}
    \label{fig:bayes_ic}
\end{figure}

For inverse problems, it is often essential to quantify the amount of ``information'' that one expects to gain through the inversion process, namely, the expected information gain (EIG). This measure is also known in the context of the optimal experimental design for linear Bayesian inverse problems as the D-optimality criterion \cite{Alexanderian21}. In the present linear Gaussian problem, 
\begin{equation}
    \Psi_{EIG} = \frac{1}{2} \log \det (\I + \Hmt),
    \label{eq:psi_kl}
\end{equation}
where $\Hmt := \prcov^{1/2}\F^*\noisecov^{-1}\F\prcov^{1/2}$ is the \textit{prior-preconditioned data-misfit} Hessian. Note that the Hessian is implicitly defined and requires two or more PDE solves per application. For large-dimensional problems, it is therefore difficult to compute the EIG explicitly. On the other hand, \algname enables efficient estimation of the EIG without explicitly computing the posterior covariance matrix. 

The magnitude of the diffusion plays an important role in the governing physics of the inverse problem. To analyze this, in \Cref{fig:norm_eigvals_compare}, we plot the eigenvalues of $\Hmt$ for $\kappa \in \{10^{-3}, 10^{-2}, 10^{-1}\}$. Observe that the rate of spectral decay of $\Hmt$ depends significantly on the diffusion coefficient $\kappa$. Namely, large values of $\kappa$ relate to diffusion-driven regimes where the Hessian tends to exhibit fast spectral decay, while smaller $\kappa$ corresponds to advection-driven regimes with slower spectral decay. Therefore, we choose to analyze the effectiveness of \algname across these scenarios. Specifically, in \Cref{fig:kappa_compare}, we display the relative errors of \algname and \funNys for the EIG estimation problem with diffusion coefficients $\kappa \in \{10^{-3}, 10^{-2}, 10^{-1}\}$. Across both advection- and diffusion-driven regimes, we observe that \algname consistently outperforms \funNys. Notice that for problems with a smaller diffusion component, both methods achieve lesser relative accuracy in trace estimation owing to the slower spectral decay of $\Hmt$. In this case, \algname better captures the trailing spectral behavior of $\Hmt$, resulting in a wider performance margin. On the other hand, in cases where there is fast spectral decay (large $\kappa$), both methods achieve high accuracy, with \algname marginally outperforming \funNys.

\begin{figure}[ht]
     \centering
     \begin{subfigure}[b]{0.3\textwidth}
         \centering
         \includegraphics[width=\textwidth]{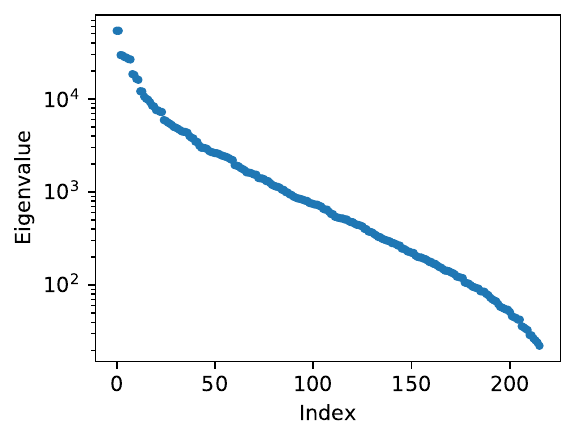}
         \caption{$\kappa = 10^{-3}$}
         \label{subfig:eigvals_1}
     \end{subfigure}
     \hfill
     \begin{subfigure}[b]{0.3\textwidth}
         \centering
         \includegraphics[width=\textwidth]{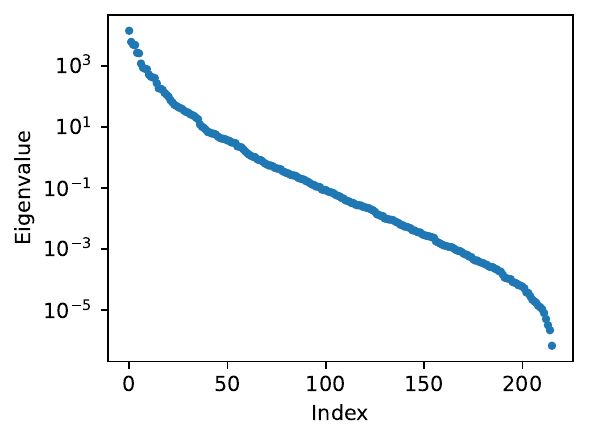}
         \caption{$\kappa = 10^{-2}$}
         \label{subfig:eigvals_2}
     \end{subfigure}
     \hfill
     \begin{subfigure}[b]{0.3\textwidth}
         \centering
         \includegraphics[width=\textwidth]{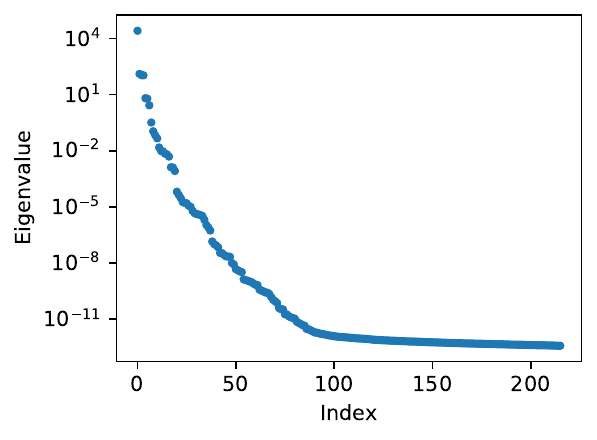}
         \caption{$\kappa = 10^{-1}$}
         \label{subfig:eigvals_3}
     \end{subfigure}
    \caption{Spectral decay of prior-preconditioned Hessian over different values of $\kappa$}
    \label{fig:norm_eigvals_compare}

     \begin{subfigure}[b]{0.3\textwidth}
         \centering
         \includegraphics[width=4cm]{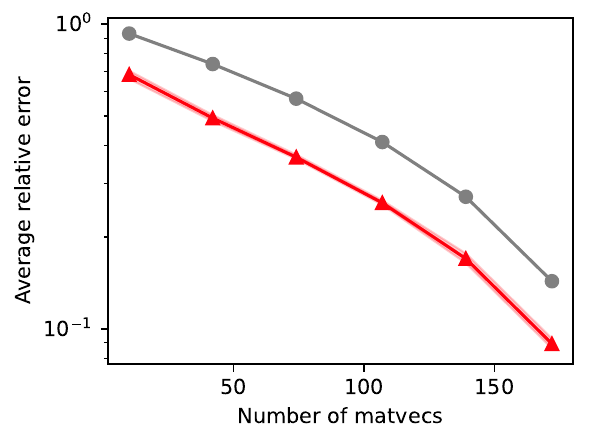}
         \caption{$\kappa = 10^{-3}$}
         \label{subfig:kappa_1}
     \end{subfigure}
     \hfill
     \begin{subfigure}[b]{0.3\textwidth}
         \centering
         \includegraphics[width=4cm]{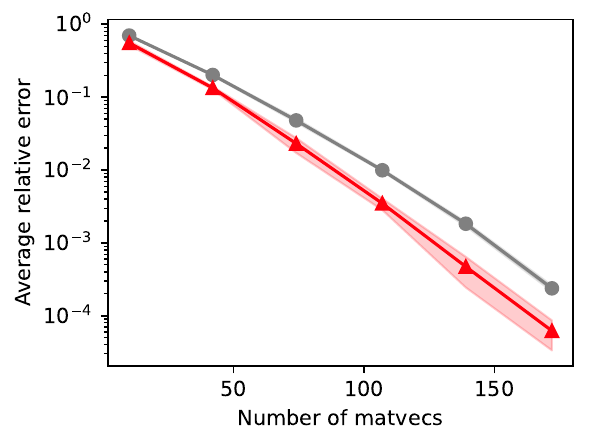}
         \caption{$\kappa = 10^{-2}$}
         \label{subfig:kappa_2}
     \end{subfigure}
     \hfill
     \begin{subfigure}[b]{0.28\textwidth}
         \centering
         \includegraphics[width=5.5cm]{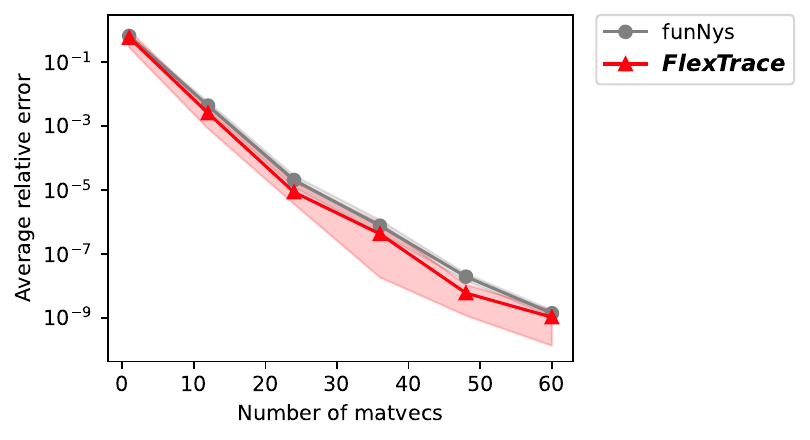}
         \caption{$\kappa = 10^{-1}$}
         \label{subfig:kappa_3}
     \end{subfigure}
    \caption{Performance of \algname in EIG estimation for different values of $\kappa$}
    \label{fig:kappa_compare}
\end{figure}

\subsection{Application to kernel methods}
\label{ssec:kernel}
In this section, we showcase the application of \algname for kernel methods. Broadly speaking, kernel methods are a family of techniques that exploit kernel functions to implicitly map input data into a high-dimensional (often infinite-dimensional) Hilbert space, thereby enabling linear models to capture intrinsically nonlinear patterns. Central to this framework is the \textit{kernel matrix}. For a set of data points $X = \{\bs{x}_1, \dots, \bs{x}_n\} \subseteq \R^d$, the kernel matrix $\mathbf{K} \in \R^{n \times n}$ is defined entry-wise as pairwise evaluations of a kernel function $\kappa: \R^d \times \R^d \to \R$, $K_{ij} = \kappa(\bs{x}_i, \bs{x}_j)$. Such matrices are important in several contexts including Gaussian processes (GP), Bayesian neural networks, determinantal point processes, and elliptical graphical models. 

We focus on Gaussian process regression, where we aim to approximate a function $f(\bs{x})$ by a Gaussian process by fitting a GP model to a collection of ``training'' data points $\{(\bs{x}_i, f(\bs{x}_i))\}_{i=1}^n$. There are several possible choices for the kernel function $k$, e.g., squared exponential kernel, quadratic kernel, Matérn kernel, etc. These kernels often depend on a set of hyperparameters that need to be optimized.
This is typically done by maximizing the log marginal likelihood of the GP model, defined as 
\[\log p(\bs{y} | X) = -\frac{1}{2} \bs{y}^\top(\K + \sigma_n^2\I)^{-1}\bs{y} - \frac{n}{2}\log(2\pi) - \frac{1}{2}\log\det(\K + \sigma_n^2\I),\]
where $\bs{y} \in \R^n$ denotes of (noisy) observations of $f(\bs{x})$ at the data points $X$ with noise variance $\sigma_n^2$. Note that explicitly computing the log marginal likelihood, particularly the log-determinant term, requires $\mathcal{O}(n^3)$ flops and $\mathcal{O}(n^2)$ storage, which can be severely limiting for large-scale problems.
Nevertheless, over the past few decades, efficient low-rank and structured sparse approximations to kernel matrices have been developed \cite{H2Pack2,H2Pack3} that enable fast approximate matvecs. We utilize the package \texttt{H2Pack} \cite{H2Pack1} that provides routines for obtaining $\mc{H}^2$ and $\mc{HSS}$ (hierarchically semi-separable) approximations of kernel matrices with linear-scaling matvec complexity. \algname can be used to compute the log-determinant of the kernel matrix without explicitly computing the kernel matrix itself.

For numerical tests, we consider the \texttt{3D\_road} elevation dataset of North Jutland, Denmark \cite{Kaul2013} from the UCI machine learning repository containing $434,874$ training data points. This data was constructed to benefit Advanced Driver Assistance Systems (ADAS) using massive aerial laser scan data spanning a region of 185 km $\times$ 130 km with an accuracy of $\pm 20$cm; see \Cref{fig:road}. We fit a Gaussian Process to the training data with added synthetic relative noise of $5\%$.

\begin{figure}[ht]
    \centering
    \includegraphics[width=0.5\textwidth]{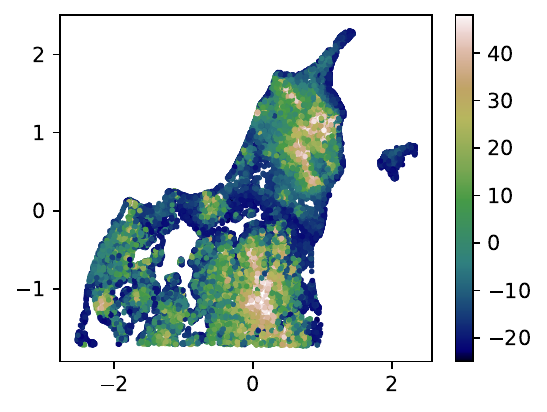}
    \caption{The \texttt{3D road} elevation dataset of North Jutland, Denmark \cite{Kaul2013}.}
    \label{fig:road}
\end{figure}

Our goal is to estimate $\log\det(\K + \sigma_n^2\I)$ implicitly. To enable computation of relative errors against the exact log-determinant, we first consider a smaller subset of the training dataset of size $10,\!000$ and scale all kernel functions by a factor of $0.1$ (fixed output variance hyperparameter).
\begin{figure}[ht]
    \centering
    \begin{subfigure}{0.35\textwidth}
        \centering
        \includegraphics[width=\textwidth]{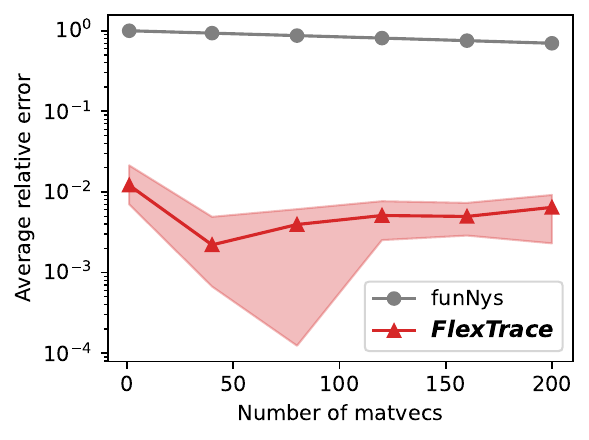}
        \caption{Squared exponential {\scriptsize ($\ell = 0.025$)}}
    \end{subfigure}
    \begin{subfigure}{0.35\textwidth}
        \centering
        \includegraphics[width=\textwidth]{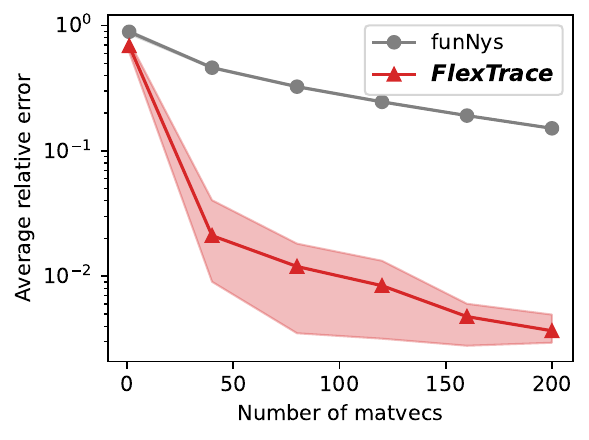}
        \caption{Rational quadratic {\scriptsize ($\ell = 0.05$, $\alpha = 0.25$)}}
    \end{subfigure}
    \begin{subfigure}{0.35\textwidth}
        \centering
        \includegraphics[width=\textwidth]{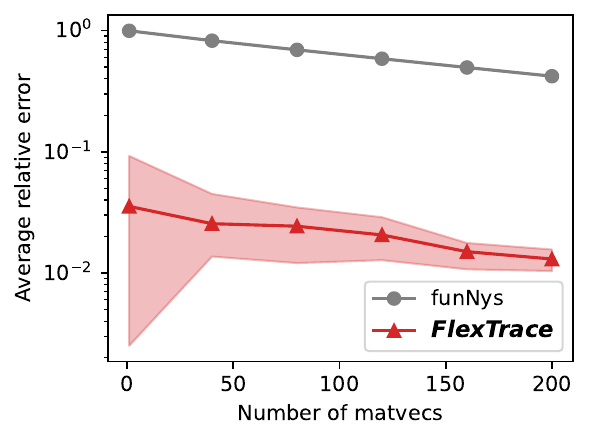}
        \caption{Matérn 3/2 {\scriptsize ($\ell = 0.10$)}}
    \end{subfigure}
    \begin{subfigure}{0.35\textwidth}
        \centering
        \includegraphics[width=\textwidth]{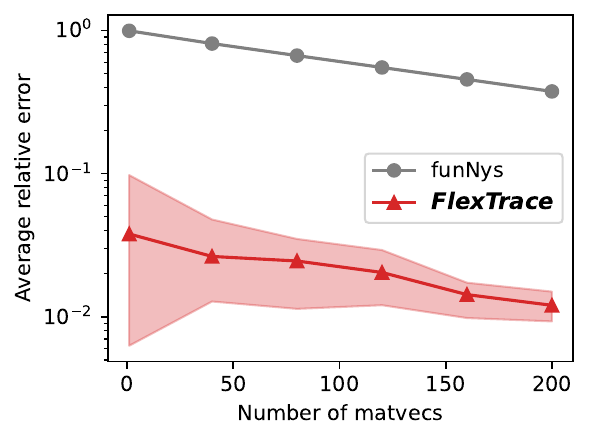}
        \caption{Matérn 5/2 {\scriptsize ($\ell = 0.10$)}}
    \end{subfigure}
    \caption{Relative error of estimation of $\log\det(\K + \sigma_n^2\I)$ using \algname.}
    \label{fig:kernel_performance}
\end{figure}
We observe in \Cref{fig:kernel_performance} that \algname consistently outperforms \funNys over the squared exponential, rational quadratic, and Matérn kernels. The strength of \algname is that it enables computationally feasible scaling of training data points. To this end, we test our approach on the full set of nearly $380,\!000$ data points. Note that since explicitly computing the true quantity for comparison becomes intractable for such large scale problems, we present simply the trace estimates obtained from \algname and \funNys in \Cref{fig:large_scale}. The results for this experiment are also not aggregated over multiple independent trials due to the high computational cost. Nevertheless, we observe that \algname matches the trace estimate obtained using \funNys with $10,\!000$ matvecs with a budget of less than half as many matvecs. This, along with our reduction in average relative error, suggests that \algname can be a powerful and effective tool for trace estimation of matrix functions in large-scale problems.

\begin{figure}[ht]
    \centering
    \includegraphics[width=7cm]{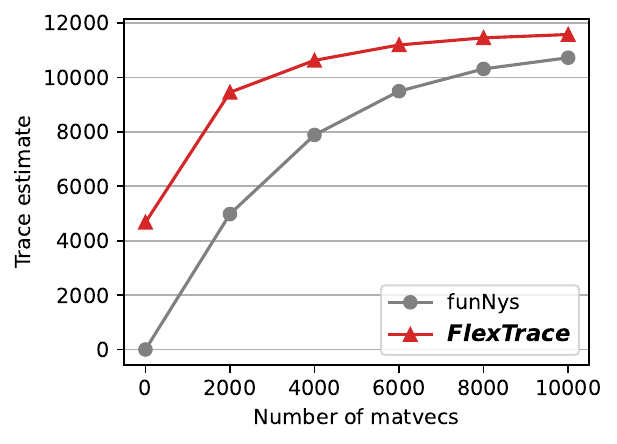}
    \caption{Application of \algname to large-scale kernel problem}
    \label{fig:large_scale}
\end{figure}

\section{Proofs of theoretical results}
\label{sec:proofs}
In this section, we detail the proofs necessary for the theoretical results presented in \Cref{ssec:key_theory}. First, we provide the proof of \Cref{thm:symmetrize} below. This is followed by the proofs of the theoretical results for \ialgname and \algname in \Cref{ssec:ialgname_proofs,ssec:algname_proofs}, respectively. Additional background and proofs of supporting results related to \Nystrom approximations are provided in \Cref{sec:app_prob_bounds}.

\begin{thmproof}{thm:symmetrize}
    Since the random vectors $\{\om_i\}_{i=1}^k$ are independent and identically distributed, 
    \begin{equation}
        \Exp[\overline{Q}] = \frac{1}{|S_k|} \sum_{\sigma \in S_k} \Exp[\Qh(\om_{\sigma(1)}, \dots, \om_{\sigma(k)})] = \Exp[\Qh].
        \label{eq:exp_equal_expr}
    \end{equation}
    Therefore, it suffices to show that the variance of $\overline{Q}$ is smaller than that of $\Qh$.
    With Jensen's inequality \cite[Equation 4]{JensenFonctions1906} applied to $\varphi(x) = x^2$, we obtain
    \[\overline{Q}^2 = \left( \frac{1}{|S_k|} \sum_{\sigma \in S_k} \Qh(\om_{\sigma(1)}, \dots, \om_{\sigma(k)}) \right)^2 \leq \frac{1}{|S_k|} \sum_{\sigma \in S_k} \left( \Qh(\om_{\sigma(1)}, \dots, \om_{\sigma(k)}) \right)^2.\]
    Taking the expectation of both sides and utilizing that $\Qh(\om_{\sigma(1)}, \dots, \om_{\sigma(k)})$ is identically distributed for $\sigma \in S_k$, we obtain
    \[\Exp[\overline{Q}^2] \leq \frac{1}{|S_k|} \sum_{\sigma \in S_k} \Exp\left[ \Qh(\om_{\sigma(1)}, \dots, \om_{\sigma(k)})^2 \right] = \Exp[\Qh^2].\]
    Observing that $\Exp[X^2] = \Var[X] + \Exp[X]^2$ for every random variable $X$ and substituting \Cref{eq:exp_equal_expr} yields $\Var[\overline{Q}] \leq \Var[\Qh]$. Finally, since $\MSE[\theta] = \Var[\theta] + \Exp[(\theta - Q)]^2$ for every estimator $\theta$ of $Q$, we obtain $\MSE[\overline{Q}] \leq \MSE[\Qh]$. Note that exchangeability of $\overline{Q}$ follows trivially by construction and hence is not proved separately.
\end{thmproof}

\subsection{Theoretical results for \ialgname}
\label{ssec:ialgname_proofs}
We first demonstrate that \ialgname is an unbiased estimator of $\tr(f(\A))$.

\begin{thmproof}{thm:bias_iFT}
    To show that $\trh_{\mathrm{iFT}}$ is an unbiased estimator of $\tr(f(\A))$, it suffices to show that $\trh_i$ is unbiased for all $1 \leq i \leq k$. Without loss of generality, consider $\trh_1$. Then,
    \begin{align}
        \Exp \left[\trh_1 - \tr(f(\A))\right] 
        &= \Exp \left[ \tr f(\AhOne) + \om_1^\top (f(\A) - f(\AhOne))\om_1 - \tr(f(\A)) \right] \nonumber \\
        &= \Exp \left[ \om_1^\top f(\A) \om_1 - \tr(f(\A)) \right] + \Exp \left[ \tr f(\AhOne) - \om_1^\top f(\AhOne) \om_1 \right]  \nonumber \\ 
        &= 0 + \Exp \left[ \tr f(\AhOne) - \om_1^\top f(\AhOne) \om_1 \right] \label{eq:proof_1}
    \end{align}
    In the final equality, the first term reduces to zero because of the unbiasedness of the Girard-Hutchinson estimator for $\tr(f(\A))$ (\Cref{thm:HutchVar}).
    For the remaining term, we use the law of total expectation,
    \begin{align*}
        \Exp \left[ \tr f(\AhOne) - \om_1^\top f(\AhOne) \om_1 \right] 
        &= \Exp_{\,\Om_{\mr 1}} \left[ \Exp_{\,\om_{1}}\left[ \tr f(\AhOne) - \om_1^\top f(\AhOne) \om_1 \ \Big| \ \Om_{\mr 1}\right] \right]
    \end{align*}
    Since $\AhOne$ does not depend on $\om_1$, we again utilize \Cref{thm:HutchVar} to observe that $\Exp_{\,\om_{1}}\left[ \tr f(\AhOne) - \om_1^\top f(\AhOne) \om_1 \ \Big| \ \Omi \right] = 0$. 
    Therefore, \Cref{eq:proof_1} reduces to
    \begin{align*}
        \Exp \left[\trh_1 - \tr(f(\A))\right] = \Exp \left[ \tr f(\AhOne) - \om_1^\top f(\AhOne) \om_1 \right] = 0.
    \end{align*}
    Thus, the average of the estimators, $\trh_{\mathrm{iFT}}$, is also an unbiased estimator of $\tr(f(\A))$.
\end{thmproof}

We now prove \Cref{thm:var_iFT}, which uses exchangeability and properties of the Girard-Hutchinson estimator to obtain a bound for the variance of \ialgname.

\begin{thmproof}{thm:var_iFT}
    The following proof mimics that of \cite[Proposition 5.3]{Epperly2023XTraceMT}. To control the variance, we note that by exchangeability,
    \begin{equation}
        \Var \left[\trh_{\rm iFT}\right]
        = \Var \Biggl[\frac{1}{k} \sum_{i=1}^{k} \trh_i\Biggr]
        = \Biggl(\frac{1}{k}\Biggr)\, \underbrace{\Var\left[\trh_1\right]}_{\sf A} 
        + \Biggl(1 - \frac{1}{k}\Biggr)\, \underbrace{\Cov\left[\trh_1, \trh_2\right]}_{\sf B}.
        \label{eq:splitAB}
    \end{equation}
    We first focus on the term {\sf A}. By the law of total variance,
    \begin{align*}
        \Var \left[ \trh_1 \right]
        &= \Exp\bigl[ \Var\big[ \trh_1 \, \big| \, \Om_{\mr 1} \big] \bigr] + \Var\bigl[ \Exp\bigl[ \trh_1 \, \big| \, \Om_{\mr 1} \bigr] \bigr].
    \end{align*}
    As observed in the proof of \Cref{thm:bias_iFT}, $\Exp\bigl[ \trh_1 \, \big| \, \Om_{\mr 1} \bigr] = \tr f(\A)$, so the second term vanishes. This leaves
    \begin{align}
        \Var \left[ \trh_1 \right] 
        &= \Exp \left[ \Var\big[ \trh_1 \, \big| \, \Om_{\mr 1} \big] \right] \nonumber \\
        &= \Exp \left[ \Var \left[ \tr f(\AhOne) + \om_1^\top (f(\A) - f(\AhOne))\om_1  \, \big| \, \Om_{\mr 1}\right]  \right]\nonumber \\
        &= \Exp \left[ \Var \left[\om_1^\top (f(\A) - f(\AhOne))\om_1  \, \big| \, \Om_{\mr 1}\right]  \right]. \label{eq:var_p_1}
    \end{align}
    The final equality in \Cref{eq:var_p_1} follows from the fact that $\tr f(\AhOne)$ is constant when conditioned on $\Om_{\mr 1}$. By \Cref{thm:HutchVar}, the variance term simplifies to
    \begin{align*}
        \textsf{A} = \Var \left[ \trh_1 \right] = 2 \Exp \left[ \| f(\A) - f(\AhOne) \|_F^2 \right].
    \end{align*}
    We turn to expanding the covariance term \textsf{B} in \Cref{eq:splitAB}. To this end, we first reformulate \textsf{B} by isolating terms that are dependent on $\om_1$ and $\om_2$. Namely, we consider the matrix $\Xh := f(\A) - f(\AhOneTwo)$, which is independent of $\om_1$ and $\om_2$, and observe that
    \begin{align}
        &\Exp\left[\left(\trh_1 - \tr f(\A) + \tr \Xh - \om_1^\top \Xh \om_1 \right)\left(\trh_2 - \tr f(\A) + \tr \Xh - \om_2^\top \Xh \om_2 \right)\right] \label{eq:p_3} \\ 
        &= \Exp\left[\left(\trh_1 - \tr f(\A)\right)\left(\trh_2 - \tr f(\A) + \tr \Xh - \om_2^\top \Xh \om_2 \right)\right] \label{eq:p_2}\\
        &= \Exp\left[\left(\trh_1 - \tr f(\A)\right)\left(\trh_2 - \tr f(\A)\right)\right] = \Cov\left[ \trh_1, \trh_2 \right]\label{eq:p_1}.
    \end{align}
    To obtain \Cref{eq:p_2} from \Cref{eq:p_3}, we condition on $\Om_{\mr 1}$ and average over $\om_1$, exploiting that $\Exp\bigl[ \trh_1 \, \big| \, \Om_{\mr 1} \bigr] = \tr f(\A)$. Similarly, to obtain \Cref{eq:p_1} from \Cref{eq:p_2}, we condition on $\Om_{\mr 2}$ and average over $\om_2$. Finally, we express {\sf B} as \Cref{eq:p_3} and apply the Cauchy--Schwarz inequality to obtain:
    \begin{align*}
        {\sf B} = \Cov\bigl[\trh_1, \trh_2\bigr]
        &= \Exp\bigl[\bigl(\trh_1 - \tr f(\A) + \tr \Xh - \om_1^\top \Xh \om_1
        \bigr)
        \bigl(\trh_2 - \tr f(\A) + \tr \Xh - \om_2^\top \Xh \om_2
        \bigr)\bigr] \\
        &\leq \Exp\left[\bigl|\trh_1 - \tr f(\A) + \tr \Xh - \om_1^\top \Xh \om_1\bigr|^2\right].
    \end{align*}
    Note that $\trh_1 - \tr f(\A)$ and $\tr \Xh - \om_1^\top \Xh \om_1$ have zero mean as a consequence of \Cref{thm:bias_iFT} and \Cref{thm:HutchVar} respectively. Therefore, 
    \[\Exp\left[\bigl|\trh_1 - \tr f(\A) + \tr \Xh - \om_1^\top \Xh \om_1\bigr|^2\right] = \Var\left[ \trh_1 - \tr f(\A) + \tr \Xh - \om_1^\top \Xh \om_1 \right]\]
    Thus,
    \begin{align*}
        {\sf B} 
        &\leq \Var\left[ \trh_1 - \tr f(\A) + \tr \Xh - \om_1^\top \Xh \om_1 \right] \\
        &= \Var\left[ \tr [f(\AhOne) - f(\AhOneTwo)] - \om_1^\top [f(\AhOne) - f(\AhOneTwo)] \om_1\right],
    \end{align*}
    where we substitute the definitions of $\trh_1$ and $\Xh$. Applying the law of total variance similarly to \Cref{eq:var_p_1}, we obtain
    \[{\sf B} \leq \Exp\left[\Var\bigl[ \om_1^\top\bigl[ f(\AhOneTwo) - f(\AhOne)\bigr]\om_1 \,\big|\, \Om_{\mr 1} \bigr]\right] = 2 \,\Exp \left[ \bigl \| f(\AhOne) - f(\AhOneTwo) \bigr \|_{\rm F}^2 \right].\]
    Combining {\sf A}, {\sf B} and bounding $1 - 1/k < 1$ completes the proof.
\end{thmproof}

\subsection{Theoretical results for \algname}~\\
\label{ssec:algname_proofs}
In this section, we prove the results for \algname presented in \Cref{ssec:key_theory}. First, we show that \algname is exchangeable. 

\begin{thmproof}{thm:flex_ex}
    To show that \algname is exchangeable, we demonstrate that \Cref{eq:FT} is a symmetrization of $\trh_1 = \tr f(\AhOne) + \om_1^\top (f(\Anys) - f(\AhOne))\om_1$. To this end, observe that $\Anys$ and $\AhOne$ are invariant to permutations of $(\om_2, \dots, \om_k)$ as a consequence of \nysprop{item:nys_invar}, and thus $\trh_1$ is also invariant to permutations of $(\om_2, \dots, \om_k)$. Therefore, to symmetrize $\trh_1$, it suffices to average over permutations $\sigma \in S_k$ such that $\sigma(1) = i$ for $1 \leq i \leq k$. This simplification yields the estimator, \Cref{eq:FT}, which, by \Cref{thm:symmetrize}, is thus exchangeable.
\end{thmproof}

\begin{thmproof}{thm:bias_FT}
    Consider the individual trace estimates $\trh_i$ from line 5 of \Cref{alg:practical}. Since these are identically distributed, $\Exp[\trh_{\mathrm{FT}}] = \Exp[\trh_i]$ for all $1 \leq i \leq k$. Without loss of generality, consider $\trh_1$. Note that 
    \begin{align}
        \Exp \ [\trh_1] &= \Exp [ \tr f(\AhOne) + \om_1^\top (f(\Anys) - f(\AhOne))\om_1 ] \nonumber \\
        &= \Exp \left[ \Exp [ \tr f(\AhOne) - \om_1^\top f(\AhOne)\om_1 \ | \ \Om_{\mr 1} ] \right] + \Exp[ \om_1^\top f(\Anys) \om_1 ] \nonumber \\
        &= \Exp[ \om_1^\top f(\Anys) \om_1 ] .
    \end{align}
    Since $f$ is operator monotone and $\AhOne \mleq \Anys \mleq \A$ (see \nysprop{item:nys_mono}), $f(\AhOne) \mleq f(\Anys) \mleq f(\A)$.
    Therefore, 
    \[0 \leq \om_1^\top f(\AhOne) \om_1 \leq \om_1^\top f(\Anys) \om_1 \leq \om_1^\top f(\A) \om_1.\]
    Taking the expectation and utilizing \Cref{thm:HutchVar},
    \[\Exp[\tr(f(\AhOne))] \leq \Exp[ \om_1^\top f(\Anys) \om_1] \leq \tr(f(\A)).\]
    Rearranging, we obtain
    \[\Exp[\tr(f(\A)) - \trh_{1}] = \tr(f(\A)) - \Exp[\om_1^\top f(\Anys) \om_1] \leq \Exp\left[\tr(f(\A) - f(\AhOne))\right].\]
    This completes the first inequality of \Cref{thm:bias_FT}. The second inequality in \Cref{thm:bias_FT}, $\Exp[\|f(\A) - f(\AhOne)\|_*] \leq (k-3)\|f(\Lm_{k\mr 2:n})\|_*$, is given in \cite[Theorem 3.9]{PerssonKressner23}. We also note that if $f$ is linear, \algname  reduces to \ialgname and is therefore unbiased; see \Cref{thm:bias_iFT}. 
 \end{thmproof}

\begin{thmproof}{thm:var_FT}
    By \Cref{thm:symmetrize}, the MSE of the symmetrized estimator, $\trh_{\rm FT}$, is bounded by the MSE of $\trh_1$. By definition of $\trh_1$,
    \begin{align*}
        \MSE[\trh_1] &\leq \Exp\left[ (\tr f(\AhOne) + \om_1^\top (f(\Anys) - f(\AhOne))\om_1 - \tr(f(\A)))^2 \right] \\
        &= \Exp\left[ (\tr(f(\A) - f(\AhOne)) - \om_1^\top (f(\Anys) - f(\AhOne))\om_1 )^2 \right] \\
        &= \Exp\left[ (\tr(\B) - \om_1^\top \G\om_1 )^2 \right],
    \end{align*}
    where we define $\B := f(\A) - f(\AhOne)$ and $\G := f(\Anys) - f(\AhOne)$. Using the property that $(\alpha-\beta)^2 \leq \alpha^2 + \beta^2$ for all $\alpha, \beta \geq 0$, 
    \[\MSE[\trh_1] \leq \Exp\left[ (\tr(\B) - \om_1^\top \G\om_1 )^2 \right] \leq \Exp\left[ (\tr(\B))^2 \right] + \Exp\left[ (\om_1^\top \G\om_1 )^2 \right].\]
    Note that $\mat{0} \mleq \G \mleq \B$ by \nysprop{item:nys_mono}. Therefore, $\tr(\B) = \|\B\|_*$ and $\om_1^\top \G \om_1 \mleq \om_1^\top \B \om_1$. So,
    \begin{align*}
        \MSE[\trh_1] \leq \Exp \left[ \|\B\|_*^2 \right] + \Exp\left[ (\om_1^\top \B\om_1 )^2 \right].
    \end{align*}
    We may simplify the right summand by using the fact that $\B$ is independent of $\om_1$. By \Cref{thm:HutchVar}, we observe that
    \begin{align*}
        \Exp_{\om_1}\left[ (\om_1^\top \B\om_1 )^2 \right] = \Var_{\om_1}\left[ \om_1^\top \B\om_1 \right] + \left[ \Exp[\om_1^\top \B\om_1] \right]^2 = 2\|\B\|_F^2 + \|\B\|_*^2.
    \end{align*}
    By the tower law of conditional expectation, $\Exp\left[ (\om_1^\top \B\om_1 )^2 \right] = \Exp\left[ \Exp_{\om_1}\left[ (\om_1^\top \B\om_1 )^2 \ | \ \B\right] \right] =\Exp\left[ 2\|\B\|_F^2 + \|\B\|_*^2 \right]$. Thus, 
    \[\MSE[\trh_1] \leq 2 \ \Exp \left[ \|\B\|_F^2 \right] + 2 \ \Exp \left[ \|\B\|_*^2 \right].\]
    Substituting the definition of $\B$ yields the desired result.
\end{thmproof}

\begin{thmproof}{thm:mse}
    The bounds for the MSE of $\trh_{\rm iFT}$ and $\trh_{\rm FT}$ immediately follow from substituting the inequality from \Cref{thm:second_mom_bd_nys} into \Cref{thm:var_iFT,thm:var_FT}, respectively. Finally, the bound for the MSE of $\trh_{\rm FN}$ follows directly from applying \Cref{thm:second_mom_bd_nys} with respect to the nuclear norm.
\end{thmproof}

\section{Conclusions}
\label{sec:concl}
In this article, we have developed \algname, a novel randomized trace estimation algorithm for the efficient estimation of $\tr(f(\A))$. The algorithm enables estimating the trace of matrix function with a randomized sketch of $\A$ by utilizing low-rank randomized \Nystrom approximations alongside classical trace estimation methods. The proposed algorithm offers advantages over traditional techniques in both performance and practicality, as \algname is hyperparameter-free, single-pass, function-agnostic, and embarrassingly parallel.
We provide theoretical guarantees for \algname and its idealized variant \ialgname for operator monotone functions and observe asymptotic decay for matrices with fast spectral decay. This is complemented with computationally scalable routines to accelerate and parallelize \algname. Our numerical results demonstrate the effectiveness of the proposed algorithm in both synthetic examples and applications. Specifically, \algname routinely outperforms comparable existing methods for matrices with different spectral properties. Moreover, we have shown that \algname accelerates the computation of quantities of interest in Bayesian inverse problems, matrix completion, and kernel methods. %

There are several interesting avenues for further research. Since \algname is a single-pass algorithm, the proposed algorithm might need to be modified for matrix functions exhibiting high nonlinearity within the spectrum. In such settings, one can consider extending \algname by exploiting properties of the function or fusing the exchangeable algorithm with multi-pass Krylov-based methods. Although the primary focus of this article has been on the application of \algname for operator monotone functions $f$, it is worth investigating the application of \algname to non-operator-monotone increasing functions. For example, in the case of $f(x) = x^2$, \algname is an unbiased estimator because of \nysprop{item:nys_quad}, and we have observed  in our numerical tests that it typically outperforms \funNys. Theoretical guarantees for the fun\Nystrom approximation for a broader class of functions is also necessary for the probabilistic analysis of \algname. Additional interesting directions include incorporating mixed-precision methods, performing stability analysis of these methods, and applying this method to different applications.

\section*{Acknowledgements}
This work was funded in part by the National Science Foundation through the award DMS-2026830 and the U.S. Department of Energy, Office of Science, Advanced Scientific Computing Research through the awards DE-SC0023188 and DE-SC0026310.
\bibliographystyle{abbrv}
\bibliography{refs}
\appendix\section{Theoretical guarantees for fun\Nystrom}~\\
\label{sec:app_prob_bounds}
In this section, we derive theoretical guarantees for the fun\Nystrom approximation \cite{PerssonKressner23}. Specifically, the primary goal of this section is to derive bounds for the expected squared norm of $f(\A)-f(\Anys)$, where $\Anys$ is a \Nystrom approximation of $\A$ (see \Cref{ssec:Nystrom}) and $f \in \mathscr{F}$. In \Cref{ssec:app_nystrom}, we list a few important properties of the \Nystrom approximation. This is followed by an outline of the theoretical approach in \Cref{ssec:app_general} and some results from random matrix theory in \Cref{ssec:app_random_theory}. Finally, we prove the key results for the fun\Nystrom approximation in \Cref{ssec:app_prob_bd}.

\subsection{Properties of the \Nystrom approximation}
\label{ssec:app_nystrom}

We first list a few useful properties of the \Nystrom approximation in \Cref{lem:nysprops} that will be useful in the proofs of our main results. 

\begin{lemma}
    The \Nystrom approximation \Cref{eq:nystrom} satisfies the following properties:\\
    \begin{enumerate}[label=(\alph*)]
        \itemsep1ex
        \item\label{item:nys_proj} \textbf{Projection:} Let $\bs{\Pi}_{\A^{1 / 2}\Om} = (\A^{1 / 2}\Om)(\A^{1 / 2}\Om)^\dagger$ denote the projection operator onto the range of $\A^{1 / 2}\Om$. Then, $\Anys = \A^{1 / 2} \bs{\Pi}_{\A^{1 / 2}\Om} \A^{1 / 2}$. 
        \item\label{item:nys_interp} \textbf{Interpolatory:} For every vector $\vec{\nu} \in \Range(\Om)$, $\Anys \vec{\nu} = \A \vec{\nu}$.
        \item\label{item:nys_invar} \textbf{Invariance:} $\Anys$ is invariant to permutations of the columns of $\Om$.
        \item\label{item:nys_mono} \textbf{Monotonicity:} $\Ahi \mleq \Anys \mleq \A$ for all $1 \leq i \leq k$.
        \item\label{item:nys_quad} \textbf{Exactness of quadratic forms:} Let $q\!: \R \to \R$ be a quadratic function. Then, 
        $\om_i^\top q(\Anys) \om_i = \om_i^\top q(\A) \om_i$ for all $1 \leq i \leq k$.\vspace{0.2cm}
    \end{enumerate}
    Note that these properties do not require $\Om$ to be a random matrix.
    \label{lem:nysprops}
\end{lemma}
\vspace{0.1cm}
\begin{proof}
    The proof of (a) is given in \cite[Lemma 1]{GittensMahoney2013Nystrom}, from which (b) immediately follows, since: 
    \[\Anys \Om = \A^{1 / 2} \bs{\Pi}_{\A^{1 / 2}\Om} \A^{1 / 2} \Om = \A\Om.\]
    Property (c) also follows from (a), since $\Range(\A^{1/2}\Om) = \Range(\A^{1/2}\Om\mat{P})$ for any permutation matrix $\mat{P}$. Finally, we show property (d). Since $\Range(\A^{1/2}\Om_{\mr i}) \subseteq \Range(\A^{1/2}\Om) \subseteq \R^n$, it holds \cite[Proposition 8.5]{HalkoMartinssonTropp2011Survey} that $\bs{\Pi}_{\A^{1/2}\Om_{\mr i}} \mleq \bs{\Pi}_{\A^{1/2}\Om} \mleq \I$ for $1 \leq i \leq k$. Therefore, by property (a), $\Ahi \mleq \Anys \mleq \A$  for $1 \leq i \leq k$. Finally, to show property (e), let $q(x) = \alpha x^2 + \beta x + \gamma$ for $\alpha, \beta, \gamma \in \R$. Observe that
    \[\om_i^\top \Anys^2 \om_i = (\Anys \om_i)^\top(\Anys \om_i) = (\A \om_i)^\top(\A \om_i) = \om_i^\top \A^2 \om_i,\]
    where we use symmetry of $\Anys$ and property (b) to substitute $\Anys \om_i = \A \om_i$ for all $1 \leq i \leq k$. Therefore,
    \[\om_i^\top q(\Anys) \om_i = \alpha (\om_i^\top \A^2 \om_i) + \beta (\om_i^\top \A \om_i) + \gamma (\om_i^\top \om_i) = \om_i^\top q(\A) \om_i,\] 
     for all $1 \leq i \leq k$. 
\end{proof}
\subsection{Outline of theoretical approach}
\label{ssec:app_general}
In this section, we provide the theoretical setting for the derivations for the fun\Nystrom approximation error that follow. To motivate our approach, we first briefly discuss an adjacent problem in randomized numerical linear algebra, which is the estimation of the optimal rank-$r$ approximation of a matrix where $ 1 \leq r \leq k$. For such problems, a randomized \Nystrom approximation obtained with $k$ random vectors \Cref{eq:nystrom} can be truncated, ignoring the $(k-r)$ trailing eigenvalues of $\Anys$. The resulting approximation is denoted here as $[\![\Anys]\!]_r$. This process of allocating more vectors for a random sketch than the desired rank and truncating is also known as \textit{oversampling} \cite{HalkoMartinssonTropp2011Survey}. 
Trace estimation techniques that rely on randomized low-rank approximations do not typically ignore any eigenvalues in the low-rank approximation. Nevertheless, we can build on the theoretical foundations developed for truncated randomized low-rank approximations. This relies on the following observation: for any unitarily invariant norm $\uinvnorm{\cdot}$, 
\begin{equation}
    \uinvnorm{f(\A)-f(\Anys)} \leq \uinvnorm{f(\A)-f(\AnysR)},
    \label{eq:oversampling}
\end{equation}
for all $1 \leq r \leq k$. This holds since $\AnysR \mleq \Anys \mleq \A$, and, 
consequently, $f(\AnysR) \mleq f(\Anys) \mleq f(\A)$, by the operator monotonicity of $f$, and the ordering of the norms follows from \cite[Lemma 3.1]{PerssonKressner23}. Therefore, to bound the error of the \Nystrom approximation $\Anys$, it suffices to bound the error of the truncated \Nystrom approximation, $\AnysR$. The following setting will be assumed for the derivations that follow.

\begin{setting}\label{setting}
    For $1 \leq r \leq k$, consider the spectral decomposition
    \begin{equation} \label{eq:spectraldecom}
        \A = \U \Lm \U^T = \begin{bmatrix} \U_1 & \U_2 \end{bmatrix} \begin{bmatrix} \Lm_1 & \\ & \Lm_2 \end{bmatrix} \begin{bmatrix} \U_1^T \\ \U_2^T \end{bmatrix},
    \end{equation}
    where $\Lm_1 = \diag(\lambda_1,\ldots,\lambda_r) \in \R^{r\times r}$,  
    $\Lm_2 = \diag(\lambda_{r+1},\ldots,\lambda_n) \in \R^{(n-r)\times (n-r)}$, $\U_1 \in \mathbb{R}^{n \times r}$, and $\U_2 \in \mathbb{R}^{n \times (n-r)}$.
    We call the partitioning of the eigenvalues/vectors in~\Cref{eq:spectraldecom}, an $r$-partition of $\A$. 
    Let $\Om \in \R^{n \times k}$ be a standard Gaussian random matrix, and
    let $\Om_1 := \U_1^T \Om \in \R^{r \times k}$ and $\Om_2 := \U_2^T \Om\in \R^{(n-r) \times k}$. We assume that $\rank(\Om_1) = r$, since this property holds almost surely (see, e.g., \cite[Lemma 3]{Yi2018OutlierDU}). Finally, we use the shorthand $\rho := k - r \geq 0$ to denote the oversampling parameter.
\end{setting}

\subsection{Some results from random matrix theory}
\label{ssec:app_random_theory}
In this section, we outline some preliminary results on Gaussian random matrices. \Cref{lem:exp_H_diff} provides expectation bounds on the fourth moments of norms of standard Gaussian random matrices under different norms. This is followed by \Cref{lem:combined_op_four} and \Cref{lem:combined_op_four_prob}, which provide properties on norms of products of random normal matrices.

\begin{lemma} %
    Let $\Hm \in \R^{r \times k}$ be a standard Gaussian random matrix, where $\rho = k - r \geq 4$. Then, 
    \begin{align*}
        &\Exp [\|\Hm^\dagger\|_{(4)}^4] = \frac{r(k-1)}{\rho(\rho-1)(\rho-3)},\\[1ex]
        &\Exp [\|\Hm^\dagger\|_F^4] = \frac{r^2\rho - 2r(r-1)}{\rho(\rho-1)(\rho-3)},\\[1ex]
        &\Exp [\|\Hm^\dagger\|_2^4] \leq \frac{e^4k^2}{(\rho+1)^3(\rho-3)}.
    \end{align*}
    \label{lem:exp_H_diff}
\end{lemma}

\begin{proof}
    The first two equalities follow from \cite[Lemma A.2]{PerssonRandomized2025}. We focus on proving the last inequality. From \cite[Proposition A.3]{HalkoMartinssonTropp2011Survey},
    \[\mathbb{P}\{\|\Hm^{\dagger}\|_2>t\}\leq C \ t^{-(\rho+1)}, \quad\text{where}\quad C = \frac{1}{\sqrt{2\pi(\rho+1)}}\left[\frac{e\sqrt{k}}{\rho+1}\right]^{\rho+1}.\]
    Therefore, we may express
    \[\mathbb{P}\{\|\Hm^{\dagger}\|_2^4>t\} = \mathbb{P}\{\|\Hm^{\dagger}\|>t^{1 / 4}\} \leq Ct^{-(\rho + 1)/4}.\]
    The integral formula for the mean of a nonnegative random variable implies that for all $E > 0$,
    \begin{align*}
        \Exp[\|\Hm^{\dagger}\|_2^4]
        = \int_0^\infty \mathbb{P}\{\|\Hm^{\dagger}\|_2^4>t\} \ dt 
        &\leq E + \int_E^\infty \mathbb{P}\{\|\Hm^{\dagger}\|_2^4>t\} \ dt \\
        &\leq E + \int_E^\infty Ct^{-(\rho + 1)/4} \ dt \\
        &\leq E + \frac{4C}{\rho-3}E^{-(\rho - 3)/4}.
    \end{align*}
    We note that the upper bound attains a minimum at $E = C^{4/(\rho+1)}$. Substituting and simplifying, we obtain 
    \begin{align*}
        \Exp[\|\Hm^{\dagger}\|_2^4] \leq \frac{e^4k^2}{(\rho+1)^3(\rho-3)} \ \left[ 2\pi(\rho+1) \right]^{-2 / (\rho+1)}.
    \end{align*}
    Substituting the fact that $\left[ 2\pi(\rho+1) \right]^{-2 / (\rho+1)} \leq 1$ for all $\rho \geq 0$ yields the desired bound.
\end{proof}

\begin{lemma}
    Consider fixed matrix $\Sm$ and independent standard Gaussian matrices $\G$ and $\Hm$ of suitable dimensions. Then,
    \[
        \Exp[\| \Sm \G \Hm^{\dagger} \|_2^4] \leq 8 \| \Sm \|_2^4 \Exp \| \Hm^{\dagger} \|_{\rm F}^4 + 16 \| \Sm \|_2^4 \Exp \| \Hm^{\dagger} \|_{(4)}^4 + 8 (\| \Sm \|_{\rm F}^4 + 2\| \Sm \|_{(4)}^4) \Exp \| \Hm^{\dagger} \|_2^4.
    \]
    \label{lem:combined_op_four}
\end{lemma}

\begin{proof}
    The following proof mostly follows from \cite[Lemma B.1]{TroppWebber2023Survey}, which states,
\begin{equation*}
    (\Exp_{\G}[\| \Sm \G \Hm^{\dagger} \|_2^4])^{1 \slash 4} 
    \leq \| \Sm \|_2 
    (\| \Hm^{\dagger} \|_{\rm F}^4
    + 2\| \Hm^{\dagger} \|_{(4)}^4)^{1 \slash 4}
    + \| \Hm^{\dagger} \|_2 
    (\| \Sm \|_{\rm F}^4
    + 2\| \Sm \|_{(4)}^4)^{1 \slash 4}.
\end{equation*}
Raising to the fourth power and using the inequality $(a+b)^4 \leq 8(a^4 + b^4)$ for $a, b \geq 0$, we get
\begin{align*}
    \Exp_{\G}[\| \Sm \G \Hm^{\dagger} \|_2^4] 
    &\leq 8 \| \Sm \|_2^4 (\| \Hm^{\dagger} \|_{\rm F}^4 + 2\| \Hm^{\dagger} \|_{(4)}^4) + 8 \| \Hm^{\dagger} \|_2^4 (\| \Sm \|_{\rm F}^4 + 2\| \Sm \|_{(4)}^4).
\end{align*}
Taking the expectation over $\Hm$, we have
\begin{align*}
    \Exp[\| \Sm \G \Hm^{\dagger} \|_2^4] 
    &= \Exp_{\Hm} [\Exp_{\G}[\| \Sm \G \Hm^{\dagger} \|_2^4]] \\
    &\leq 8 \| \Sm \|_2^4 \Exp_{\Hm} [\| \Hm^{\dagger} \|_{\rm F}^4 + 2\| \Hm^{\dagger} \|_{(4)}^4] + 8 (\| \Sm \|_{\rm F}^4 + 2\| \Sm \|_{(4)}^4) \Exp_{\Hm} [\| \Hm^{\dagger} \|^4] \\
    &= 8 \| \Sm \|_2^4 \Exp[\| \Hm^{\dagger} \|_{\rm F}^4] + 16 \| \Sm \|_2^4 \Exp [\| \Hm^{\dagger} \|_{(4)}^4] + 8 (\| \Sm \|_{\rm F}^4 + 2\| \Sm \|_{(4)}^4) \Exp [\| \Hm^{\dagger} \|^4].
\end{align*}
\end{proof}

\begin{lemma}
    Consider a fixed matrix $\Sm$ and independent standard Gaussian matrices $\G$ and $\Hm$ of suitable dimensions, where $\Hm \in \R^{r \times k}$ with $\rho = k - r \geq 4$. Then,
    \[\Exp [\|\Sm\G\Hm^\dagger\|_2^4] \leq \frac{8r(r+2)}{(\rho-1)(\rho-3)}\| \Sm \|_2^4 +  \frac{8e^4k^2}{(\rho+1)^3(\rho-3)} (\| \Sm \|_{\rm F}^4 + 2\| \Sm \|_{(4)}^4).\]
    \label{lem:combined_op_four_prob}
\end{lemma}
\begin{proof}
    The result follows from substituting the bounds for $\Exp [\|\Hm^\dagger\|_F^4]$, $\Exp [\|\Hm^\dagger\|_{(4)}^4]$, and $\Exp \|\Hm^\dagger\|_2^4]$ from \Cref{lem:exp_H_diff} into \Cref{lem:combined_op_four}.
\end{proof}

\subsection{Probabilistic Bounds for Low-Rank Approximation}
\label{ssec:app_prob_bd}
In this section, we provide a probabilistic bound for the squared approximation error of fun\Nystrom method. This is first aided by the structural bound, \Cref{lem:nystr_expand}, which decomposes this quantity as the sum of a deterministic and probabilistic quantity.
\begin{lemma}
    Consider the $r$-partitioning of $\A$ under \Cref{setting}. Then,
    \begin{equation*}
        \uinvnorm{f(\A)-f(\AnysR)}^2 \leq 2\uinvnorm{f(\Lm_2)}^2 + 2\uinvnorm{f((\Om_2 \Om_1^{\dagger} )^\top \Lm_2 (\Om_2 \Om_1^{\dagger} ))}^2,
    \end{equation*}
    where $\uinvnorm{\cdot}$ denotes any unitarily-invariant norm.
    \label{lem:nystr_expand}
\end{lemma}
\begin{proof}
    Follows directly from \cite[Lemma 3.13]{PerssonKressner23} and the fact that $(\alpha+\beta)^2 \leq 2\alpha^2 + 2\beta^2$ for all nonnegative scalars $\alpha, \beta \geq 0$.
\end{proof}

We would like to obtain bounds in expectation for the second term of the bound in \Cref{lem:nystr_expand}. We prove such a bound in \Cref{lem:second_mom_bd} following which, we present the main result of this section, \Cref{thm:second_mom_bd}. To this end, we first state an elementary property on operator monotone functions in \Cref{lem:opMonotoneProp}.

\begin{lemma}
    For operator monotone function $f:[0,\infty) \mapsto [0,\infty)$,
    \[f(tx) \leq \max\{t, 1\} f(x),\]
    for all $t, x \geq 0$.
    \label{lem:opMonotoneProp}
 \end{lemma}
 \begin{proof}
    This result follows directly as a consequence of \cite[Lemmas 2.1-2.2]{PerssonMeyerMusco25}.
 \end{proof}

\begin{lemma}
    Consider the $r$-partitioning of $\A$ under \Cref{setting} with oversampling $\rho = k - r \geq 4$. Then,
    \begin{align*}
        \Exp\left[ \uinvnorm{f((\Om_2 \Om_1^{\dagger} )^\top \Lm_2 (\Om_2 \Om_1^{\dagger}))}^2 \right] \leq &\Bigg[ 1 + \frac{8r(r+2)}{(\rho-1)(\rho-3)}\| \Lm_2 \|_2 \\ &+ \frac{8e^4k^2}{(\rho+1)^3(\rho-3)} (\| \Lm_2^{1 / 2} \|_{*}^2 + 2\| \Lm_2 \|_*) \Bigg] \uinvnorm{f(\Lm_2^{1 / 2})}^2.
    \end{align*}
    \label{lem:second_mom_bd}
\end{lemma}

\begin{proof}
    Let $\Z := \Om_2 \Om_1^{\dagger} \in \R^{(n-r) \times k}$. Note that for any SPSD matrices $\A$ and $\D$, $\D^{1 / 2}\A\D^{1 / 2} \mleq \|\A\|_2 \D$.
    Since $f$ is operator monotone, 
    \begin{align*}
        f(\Z^\top \Lm_2 \Z) = f(\Z^\top \Lm_2^{1 / 4}  \Lm_2^{1 / 2}  \Lm_2^{1 / 4}\Z) 
        \mleq\, f(\|\Lm_2^{1 / 4}\Z\|_2^2 \Lm_2^{1 / 2}).  
    \end{align*}
    Since $\Lm_2$ is diagonal with nonnegative entries, we can apply \Cref{lem:opMonotoneProp} to obtain,
    \[f(\Z^\top \Lm_2 \Z) \mleq f(\|\Lm_2^{1 / 4}\Z\|_2^2 \Lm_2^{1 / 2}) \mleq \max\{\|\Lm_2^{1 / 4}\Z\|_2^2, 1\} f(\Lm_2^{1 / 2}).\]
    Therefore, by \cite[Lemma 3.1]{PerssonKressner23},
    \[\uinvnorm{f(\Z^\top \Lm_2 \Z)} \leq \max\{\|\Lm_2^{1 / 4}\Z\|_2^2, 1\}\uinvnorm{f(\Lm_2^{1 / 2})}.\]
    Squaring both sides, bounding the maximum by the sum, and taking the expectation, we obtain
    \[\Exp\left[ \uinvnorm{f(\Z^\top \Lm_2 \Z)}^2 \right] \leq \left( \Exp\left[ \|\Lm_2^{1 / 4}\Z\|_2^4 \right] + 1 \right)\uinvnorm{f(\Lm_2^{1 / 2})}^2.\]
    Applying \Cref{lem:combined_op_four_prob}, 
    \[\Exp \left[ \|\Lm_2^{1 / 4}\Om_2 \Om_1^{\dagger}\|_2^4 \right] \leq \frac{8r(r+2)}{(\rho-1)(\rho-3)}\| \Lm_2 \|_2 +  \frac{8e^4k^2}{(\rho+1)^3(\rho-3)} (\| \Lm_2^{1 / 2} \|_{*}^2 + 2\| \Lm_2 \|_*),\]
    and we obtain the desired result.
\end{proof}

\begin{theorem}
    Consider the $r$-partitioning of $\A$ under \Cref{setting} with oversampling $\rho = k - r \geq 4$. Then,
    \begin{align*}
        \Exp\left[ \uinvnorm{f(\A)-f(\AnysR)}^2 \right] \leq 2\uinvnorm{f(\Lm_2)}^2 + &\Bigg[ 2 + \frac{16r(r+2)}{(\rho-1)(\rho-3)} \|\Lm_2\|_2 \\&\quad +   \frac{16e^4k^2}{(\rho+1)^3(\rho-3)} (\|\Lm_2^{1/2}\|_*^2 + 2\|\Lm_2\|_*) \Bigg] \uinvnorm{f(\Lm_2^{1 / 2})}^2.
    \end{align*}
    \label{thm:second_mom_bd}
\end{theorem}
\begin{proof}
    This result follows from combining \Cref{lem:second_mom_bd,lem:nystr_expand}.
\end{proof}
\begin{thmproof}{thm:second_mom_bd_nys}
    This result follows from substituting $r = k - 4$ in \Cref{thm:second_mom_bd} and combining with \Cref{eq:oversampling}. 
\end{thmproof}

\end{document}